\documentclass[onecolumn]{autart}
\usepackage{}    

\usepackage{tikz} 
\usetikzlibrary{shapes,arrows,chains,backgrounds,fit,decorations.pathreplacing}
\usepackage[europeanresistors,americaninductors]{circuitikz}
\usetikzlibrary{shapes,arrows}
\usepackage{amsmath,bm,times}

\usepackage{graphicx}
\usepackage{amsmath,amssymb,amsfonts}
\usepackage{mathrsfs}
\usepackage{epstopdf}
\usepackage{subfigure}
\usepackage{mathrsfs}
\usepackage{bbm}    

\usepackage{lineno,hyperref}

\usepackage{graphicx}           
\usepackage{dcolumn}            
\usepackage{bm}                 

\usepackage{graphics}           
\usepackage{epsfig}             
\usepackage{times}              
\usepackage{amsmath}
\usepackage{amssymb}
\usepackage{amsfonts}
\usepackage{fancyhdr}
\usepackage{color}
\usepackage{subfigure}
\usepackage{enumerate}
\usepackage{cite}
\usepackage{wrapfig}
\usepackage{url}
\allowdisplaybreaks[4] 

\newtheorem{theorem}{Theorem} %
\newtheorem{lemma}{Lemma}
\newtheorem{corollary}{Corollary}
\newtheorem{proposition}{Proposition}

\newtheorem{definition}{Definition}

\newtheorem{remark}{Remark}
\newenvironment{proof}{{\it Proof:\enspace}}{\hfill $\blacksquare$\par}

\newenvironment{proof of Prop.ODE}{{\it Proof of Proposition~\ref{Prop.ODE}:\enspace}}{\hfill $\blacksquare$\par}

\newenvironment{proof of PDE}{{\it Proof of Proposition \ref{Prop.8}:\enspace}}{\hfill $\blacksquare$\par}

 \usepackage{etoolbox}
\makeatletter
\let\original@makecaption\@makecaption
\long\def\small@makecaption#1#2{%
	\vskip\abovecaptionskip
	\sbox\@tempboxa{\fontsize{8}{11.5}\selectfont {\fontsize{8.5}{11.5}\selectfont #1:} #2}%
	\ifdim \wd\@tempboxa >\hsize
	\fontsize{8}{11.5}\selectfont {\fontsize{8.5}{11.5}\selectfont #1}\quad #2\par
	\else
	\global \@minipagefalse
	\hb@xt@\hsize{\hfil\box\@tempboxa\hfil}%
	\fi
	\vskip\belowcaptionskip}

\newcommand{\smallcaption}[1]{%
	\let\@makecaption\small@makecaption
	\caption{#1}%
	\let\@makecaption\original@makecaption
}
\makeatother

\begin{document}

%
\begin{frontmatter}

\title{Local integral input-to-state stability for non-autonomous infinite-dimensional systems} 

\author{Yongchun Bi$^{1,2}$},
\author{Panyu Deng$^{2}$},
\author{Jun Zheng$^{2,3}$}\ead{zhengjun2014@aliyun.com},
\author{Guchuan Zhu$^{3}$}

\address{$^{1}${State Key Laboratory of Rail Transit Vehicle System}, Southwest Jiaotong University, Chengdu, China\\
	$^{2}${School of Mathematics}, Southwest Jiaotong University, Chengdu, China\\
                $^{3}$Department of Electrical Engineering, Polytechnique
                Montr\'{e}al,
                Montreal,  Canada}  

\begin{abstract}                           
 In this paper, we   prove  comparison principles for nonlinear  differential equations with time-varying coefficients  and develop Lyapunov analytical tools  for  the  integral input-to-state stability (iISS) analysis of nonlinear non-autonomous  infinite-dimensional systems,  which   involve  nonlinearities satisfying  a superlinear growth,   {bringing} difficulties to the iISS {analysis.}
 Specifically,   our approach starts by establishing several forms of comparison principles   for a wide range of ordinary differential equations having time-varying coefficients and superlinear terms, paving the way to conduct  iISS assessment for general nonlinear non-autonomous  infinite-dimensional systems within the Lyapunov stability framework. Then, by using  the comparison principles,  we  prove a  local  {iISS}  {(LiISS)} Lyapunov theorem for the nonlinear non-autonomous  infinite-dimensional systems in the framework of Banach spaces.  {Furthermore,}
 we   provide  sufficient conditions of the existence of a local iISS Lyapunonv functional (LiISS-LF) and construct LiISS-LFs for the systems in the framework of Hilbert spaces. Finally,   we    preset two examples to illustrate the proposed {Lyapunov} method for the LiISS analysis: one is  to show how to obtain the LiISS of a nonlinear  finite-dimensional system  with time-varying coefficients and superlinear terms under linear state feedback control law while another one is to show how to
 employ the interpolation inequalities to handle superliner terms and establish the LiISS-LF for a class of multi-dimensional parabolic equations with space-time-varying coefficients. To demonstrate the validity of the results, numerical experiments are also conducted to verify the LiISS of these two classes of systems.
\end{abstract}

\end{frontmatter}
\endNoHyper

\section{Introduction}\label{sec1}

The notion of input-to-state stability (ISS)   was initially introduced  by Sontag for nonlinear finite-dimensional  systems \cite{sontag1990further}, which provides a crucial tool for stability analysis   by characterizing  the impact of external inputs on the  system dynamics.
In the past few decades,   the ISS  theory has witnessed rapid development in the  nonlinear and robust control of finite-dimensional systems; see, e.g.,    \cite{Aghaeeyan2020,freeman2008robust,Khalil2001non,zhu2020} for the robust stability analysis of nonlinear systems;  \cite{andrieu2009unifying,arcak2001nonlinear,chen2012CNSNS,zhuang2025} for the design of nonlinear observers;  and \cite{dashkovskiy2007iss,dashkovskiy2010small,jiang1994small} for
small-gain theorems  of large-scale networks,   just to cite a few.

It is worth noting that  in \cite{sontag1998comments}, Sontag   proposed a new stability concept, which is weaker than the ISS,  to extend $L^2$ stability to nonlinear systems, namely, the  integral input-to-state stability (iISS). Allowing for unbounded inputs of finite energy, the iISS expands the theoretical framework of   stability analysis by using an integral norm for inputs while employing a supremum norm   for states,  thereby encompassing nonlinear systems that do not meet ISS requirements.
Especially,
various characterizations of iISS  have been provided in \cite{angeli2000IEEETAC} based on dissipation inequalities and other necessary and sufficient conditions involving alternative and nontrivial characterizations, thus laying a foundation for the theory and applications of iISS for nonlinear systems.

Over the past two decades,  the  iISS   of finite-dimensional systems has been extensively extended to infinite-dimensional  systems {\cite{damak2021input,damak2023input,Heni:2025,Hosfeld2022,Karafyllis:book,Lin2018,JoseL2023,mironchenko2016MCRF,Mironchenko2015SIAM,mironchenko2020SIAM,pepe2006, Zheng2018MCSS,zheng2021ALM,zheng2018input,Zheng2018CDC}.}  For infinite-dimensional systems having a   specific form  and time-invariant coefficients, the iISS   has been extensively studied {\cite{Heni:2025,Karafyllis:book,Lin2018,Mironchenko2015SIAM,pepe2006,zheng2018input,Zheng2018MCSS,Zheng2018CDC}}. Especially,   for  partial differential equations (PDEs) with time-invariant coefficients,    Lyapunov-based methods have been developed  to analyze the iISS of  nonlinear parabolic PDEs and   a small-gain criterion has been established for interconnected iISS systems, with applications to reaction-diffusion equations, in \cite{Mironchenko2015SIAM}; and by constructing Lyapunov functionals, the iISS  of  a class of semi-linear parabolic PDEs and {coupled} PDEs   were proved in  \cite{Zheng2018CDC,zheng2018input}  and \cite{Zheng2018MCSS}, respectively, in the presence of boundary disturbances.

For autonomous infinite-dimensional systems under an abstract form, the authors of \cite{mironchenko2016MCRF,mironchenko2020SIAM}  investigated  the iISS of the following  bilinear system: 
\begin{align*}
	\dot{x}(t) = &   A x(t) + B u(t) + C(x(t), u(t)),   t> 0,\notag\\
	{x(0)=}& {x_0,}
\end{align*}
where   $x \in X$ is the state, $x_0 \in X$ is the initial state, and $u \in U$ is the input, $A: D(A) \subseteq X \to X$ is the generator of a strongly continuous semigroup $(T(t))_{t \geq 0}$, $B \in \mathcal{L}(U, X)$ is a bounded   linear input operator, and $C: X \times U \to X$ is a bilinear map,   where  $X$  and $U$ are both {Banach spaces}.
Especially,   by using Gronwall's inequality and logarithmic transformations, the authors provided an iISS characterization for the system, demonstrating that the {global} asymptotic stability  at zero uniformly with
respect to {(w.r.t.)} state  (0-UGASs) implies the iISS. Furthermore, in Hilbert spaces,  the authors have successfully constructed  iISS Lyapunov functionals (iISS-LFs)    to assess the iISS; see \cite{mironchenko2016MCRF,mironchenko2020SIAM}.       It is worth noting that, based on semigroup theory and admissibility conditions,  the iISS criteria for bilinear systems with unbounded control operators  were established   in \cite{Hosfeld2022},  and iISS characterizations   were provided    for linear systems with unbounded control operators   within the framework of Orlicz spaces  in \cite{Jacob2018}, respectively.

For non-autonomous infinite-dimensional systems under an abstract form,   the  first   attempt to   investigate  the iISS is due to \cite{damak2021input}, where the authors  {studied} the following non-autonomous bilinear infinite-dimensional system:
\begin{align*}
	\dot{x} =& A(t)x + B(t)u + G(t, x, u),  t \geq t_0 \geq 0, \\
	x(t_0) =& x_0,
\end{align*}
where 
$A(t)$  is a time-varying linear operator, $B(t)$ is a bounded linear operator,
and $G:[t_0,+\infty)\times X\times U\rightarrow X$ is a nonlinear functional, where  $X$ is a {Banach space} and $U$ is a  normed linear space. Especially,  the equivalence between the iISS and the  0-UGASs property    was established, and  an explicit method for constructing an iISS Lyapunov functional   in Hilbert spaces was provided.
In \cite{damak2023input} {and \cite{Heni:2025}, the authors provided   iISS  (and ISS) characterizations}  for non-autonomous {bilinear and linear   infinite-dimensional   systems, respectively.  Moreover,   in \cite{Heni:2025}, the authors   constructed a coercive  iISS-LF  for a certain class of time-varying semi-linear evolution equations with unbounded linear operator $A(t)$ based on the  application of the 0-UGASs property of the system. In addition,
in}    \cite{JoseL2023}, the authors studied a class of infinite-dimensional time-varying control systems, demonstrating the equivalence among the iISS, 
0-UGASs, and the uniform bounded-energy input/bounded state  properties.

It is noteworthy that  the  method adopted in {\cite{mironchenko2016MCRF,mironchenko2020SIAM,Hosfeld2022,damak2021input,damak2023input,Heni:2025,JoseL2023}} for obtaining the iISS  highly relies on verifying the 0-UGASs property of the  considered systems.    Although proving the 0-UGASs property seems to be relatively straightforward for autonomous infinite-dimensional systems by using the semigroup theory, it becomes  rather challenging for non-autonomous infinite-dimensional systems{, particularly, PDEs with time-varying coefficients.}
The main reason is that, unlike   autonomous infinite-dimensional systems, the non-autonomous infinite-dimensional systems are associated to a family of time-dependent operators $A(t)$,   which precludes  the existence of a single infinitesimal generator. This inherent structure makes the cornerstone results of semigroup theory, such as the Hille-Yosida theorem,  inapplicable. 
Consequently, the 0-UGASs property  must be ascertained through more complex  tools  for nonlinear systems such as  the construction of time-dependent Lyapunov functionals, moving beyond the classical spectral-based approach for linear systems.

%

It is also worth noting that  all of  the nonlinear terms considered in {\cite{mironchenko2016MCRF,mironchenko2020SIAM,Hosfeld2022,damak2021input,damak2023input,Heni:2025,JoseL2023}} have    natural growth condition  { w.r.t. } the state $x$. For instance, in \cite{mironchenko2016MCRF}, the nonlinear term $C: X \times U \to X$ was supposed to satisfy the following condition:
\begin{align*}
	\|C(x, u)\|_X \leq k\|x\|_X   \xi(\|u\|_U),   \forall x \in X, u \in U
\end{align*}
with a positive constant $k$ and a function $\xi \in \mathcal{K}$;   the nonlinear term $G$ in {\cite{damak2021input,Heni:2025}} was supposed to satisfy the following condition:
\begin{align*}
	{\|G(t,x,u)\|_X}\leq k \|x\|_X \phi(\|u\|_U),{\forall x\in X,u\in U,t\geq t_0,}
\end{align*}
where $k$ is a positive constant and $\phi$ is a  {$\mathcal{K}$-function}; and in \cite{JoseL2023},   the condition for the considered nonlinear term $f:\mathbb{R}_{\geq 0}\times X \times U \rightarrow X$   is included by the following    growth condition:
\begin{align*}
	\|f(t,x,u)\|_{{X}}\leq k(1+  \|x\|_{{X}} )\gamma(\|u\|_U),\forall x\in X, u\in U,
\end{align*}
where $k$ is a non-negative constant,  $\gamma$ is a  {$\mathcal{K}_\infty$-function}, $X$ is a Banach space, and $U$ is a  normed linear space.  Under the structural conditions imposed in {\cite{mironchenko2016MCRF,mironchenko2020SIAM,Hosfeld2022,damak2021input,damak2023input,Heni:2025,JoseL2023}},     all the considered nonlinear   maps are from the state space $X$ to the same space $X$ w.r.t. $ x$  and hence, various iISS analysis tools and results   for  linear   infinite-dimensional  systems  can be extended to  these systems, thereby making the problem more tractable.

However, many infinite-dimensional systems used to describe physical phenomena or dynamic behavior often have nonlinear terms with superlinear growth, such as in the models of population growth \cite{Murray2003}, tumor angiogenesis \cite{Hillen2009},  porous media flow~\cite{Vazquez2007}, etc.  In contrast to the nonlinear terms with natural growth {\cite{mironchenko2016MCRF,mironchenko2020SIAM,Hosfeld2022,damak2021input,damak2023input,Heni:2025,JoseL2023}}, when a nonlinear function  has a superlinear growth property in the system, it  often maps the state from  a smaller space  $X_\alpha$, rather than   $X$, to the space $X $,    where   $X_{\alpha}$ with $ \alpha \in (0,1)$  denotes the fractional power space  associated with the considered differential operator $A(t)$, thereby,   bringing difficulties to iISS analysis.
In particular, PDEs  with general superlinear    terms may not be characterized via  global iISS, as the solution can blow-up    in finite time   under large initial values (see, e.g.,  \cite[Chap. 5]{bebernes2013} and \cite{merle1997}) unless special structural conditions are imposed; see, e.g.,  the monotonicity or sign conditions of superlinear terms for  ensuring the   iISS  of parabolic  PDEs   \cite{zheng2018input}.  

Notably, for autonomous infinite-dimensional {systems} with superlinear terms, local ISS   has been studied in 
\cite{mironchenko2016local}.   Especially, it was shown  that when   nonlinear terms satisfy a sort of uniform continuity condition  w.r.t. external inputs, {the local asymptotic
	stability at zero uniformly w.r.t. state (0-UASs) is equivalent to the local input-to-sate stability of the system. In addition, a}  local ISS Lyapunov theorem and the  inverse theorem have been proved in \cite{mironchenko2016local}, namely, the existence of a local ISS Lyapunov functional  is equivalent to the local ISS of the system.
However, to the best  knowledge of the authors, results on the local iISS (LiISS) of non-autonomous (and even autonomous) infinite-dimensional systems with generic nonlinear terms have not been reported in the existing literature yet. {Moreover, as mentioned earlier, verifying the 0-UASs for non-autonomous
	infinite-dimensional systems under a general form remains a challenge task.}


In this paper, we investigate the  LiISS   for   non-autonomous infinite-dimensional systems with superlinear  terms.  {In particular,    the considered nonlinear term, denoted by $F(t,x,u)$,  is allowed to be with a growth as
	\begin{align*}
		\| F(t, x,u) \|_{{X}} \leq k (t)(1 + \| {x} \|_{ X _{\alpha}})+ \xi(\|u\|_U),\forall x\in X_{\alpha},u\in U,t\geq 0,
	\end{align*}
	or
	\begin{align*}
		\| F(t, x,u) \|_{{X}} \leq k (t)\left(1+ \| {x} \|_{ X  }^m\right)   \xi(\|u\|_U),\forall x\in X_{\alpha},u\in U,t\geq 0,
	\end{align*}
	where    $k $ is a non-negative	locally integrable function, $m>1$ is a constant, and $\xi$ is a $\mathcal{K}$-function.

	To avoid the difficulties arising from the ``hard-to-verify {0-UASs (or 0-UGASs)}'' in verifying  the iISS criteria for non-autonomous infinite-dimensional systems  and   overcome the impact of superlinear terms on stability analysis, we develop Lyapunov analytical tools for  establishing the  LiISS  of the considered systems.
	More
	precisely, we first prove     comparison principles for a wide range  of nonlinear ordinary differential equations (ODEs)  having   time-varying coefficients, providing   fundamental and unified tools for the LiISS assessment of non-autonomous infinite-dimensional systems in the framework of  Lyapunov stability theory.
	Then, for a class of non-autonomous infinite-dimensional systems, we  prove an LiISS Lyapunov theorem, i.e., the existence of an  LiISS Lyapunov functional (LiISS-LF) implies the LiISS of the system, in the framework of Banach spaces. Furthermore, we provide sufficient conditions to ensure the existence of an LiISS-LF and construct LiISS-LFs for a class of non-autonomous infinite-dimensional systems in the framework of Hilbert spaces. Finally, as an application of the Lyapunov analytical tools developed in this paper, we establish the LiISS and conduct numerical simulations for a class of ODE systems, and multi-dimensional parabolic PDEs, respectively,  with superlinear terms and space-time-varying coefficients. In particular,  for PDEs, we show how to employ the  interpolation inequalities    to      {transform the  $X_{\alpha}$-norm of  states into a high-order term of the $X$-norm, thereby effectively handling superlinear terms and, consequently, establishing the LiISS   of   the {system}.}

	Overall, the main contributions of this paper are threefold:
	\begin{enumerate}[(i)]
		\item   We  prove  generalized comparison principles for nonlinear ODEs under a general form to capture local  properties of the states,  providing a  critical   tool for establishing   the LiISS of non-autonomous infinite-dimensional systems without verifying the {0-UASs (or 0-UGASs)} via  the semigroup theory of linear operators.
		\item   {We develop generic  Lyapunov   tools for    the LiISS  analysis of  non-autonomous infinite-dimensional systems by simultaneously considering  the effect of  time-varying coefficients, superlinear terms, and the external inputs on the systems' stability,  which can be used in a wide range of scenarios.}
		
		\item    To validate the universality and effectiveness of the developed method, we     verify the LiISS for a class of nonlinear ODEs  and PDEs, respectively,  with  time-varying coefficients. Especially, {for PDEs, we present the application of the  interpolation inequalities in the stability analysis. The   importance  of this technique lies in providing an  insightful approach to managing superlinear terms, which is essential for  proving the LiISS of nonlinear PDEs.}
	\end{enumerate}
	
	In the rest of the paper, we first introduce  some basic notations used in this paper. In Section \ref{comparison principle}, we prove several forms of {comparison principles} for nonlinear  ODEs with time-varying coefficients,   providing crucial Lyapunov analytical tools for the LiISS analysis of
	non-autonomous infinite-dimensional systems.  In Section \ref{Section-2}, we consider the LiISS for a class of  non-autonomous infinite-dimensional systems under an abstract form.  Using the comparison principles for ODEs, we first prove an LiISS Lyapunov  theorem for the  non-autonomous infinite-dimensional systems   in the  framework of   Banach spaces. Then, we provide sufficient conditions to ensure the existence of an LiISS-LF  for  the non-autonomous infinite-dimensional systems in in the  framework of   Hilbert spaces.   In Section \ref{sec:5},  we present two examples to illustrate the application of the LiISS Lyapunov  theorem to an ODE system and a PDE system, respectively. More precisely, in Section \ref{sec:5.1}, we consider the stabilization problem for an ODE system with time-varying coefficients and superlinear terms via a  linear state feedback control. By  using the  LiISS Lyapunov  theorem, we prove the LiISS of the ODE system under closed loop. In  Section \ref{sec:5.2}, we verify the LiISS for a class of multi-dimensional
	parabolic PDEs with space-time-varying coefficients and superlinear terms by employing the proposed Lyapunov method. In
	particular, we show how to handle the superlinear terms by using the  interpolation inequalities.  Numerical experiments are conducted to verify the LiISS of these ODE and PDE systems. Finally, some concluding remarks are given in Section \ref{conclusion}.
	
	
	\textbf{Notation}\quad Let ${\mathbb{N}_0}:=\{0,1,2,...\}$, ${\mathbb{N}}:=\mathbb{N}_0\setminus\{0\}$, $\mathbb{R}:=(-\infty,+\infty)$, $\mathbb{R}_{>0}:=(0,+\infty)$, and $\mathbb{R}_{\geq 0}:=[0,+\infty)$.  Let $\mathbb{R}^N$ (${N}\in\mathbb{N}$) denote the $N$-dimensional Euclidean space with the norm $|\cdot|$.
	
	We use the following classes of comparison functions:
	\begin{align*}
		\mathcal{K}:=&\left\{\gamma:\mathbb{R}_{\geq0}\rightarrow\mathbb{R}_{\geq0}\big|\gamma\ \mbox{is continuous,\ strictly\ increasing,\ and}\ \gamma(0)=0\right\},\\
		\mathcal{K}_{\infty}:=&\left\{\gamma \in \mathcal{K}\big|\gamma\ \mbox{is unbounded}\right\},\\
		\mathcal{L}:=&\left\{\gamma:\mathbb{R}_{\geq0}\rightarrow\mathbb{R}_{\geq0}\big|\gamma\ \mbox{is continuous,\ strictly\ decreasing,\ and}\ \lim_{t\rightarrow\infty }\gamma(t)=0 \right\},\\
		\mathcal{K}\mathcal{L}:=&\left\{\beta: \mathbb{R}_{\geq0}\times\mathbb{R}_{\geq0}\rightarrow\mathbb{R}_{\geq0}\big|\beta\ \mbox{is continuous},\ \beta(\cdot, t) \in \mathcal{K},\forall t\in\mathbb{R}_{\geq0};\beta(r, \cdot) \in \mathcal{L} ,\forall r\in \mathbb{R} _{>0}\right\}.
	\end{align*}

	For   $T\in \mathbb{R}_{>0}$,  let ${C([0,T]};\mathbb{R}_{\geq 0}):=\big\{v: {[0,T]} \rightarrow \mathbb{R}_{\geq 0}\big|  v$ is continuous on $[0,T]\big\}$.  Let $C(\mathbb{R}_{\geq 0};\mathbb{R}_{\geq 0}):=\big\{v: {\mathbb{R}_{\geq 0}} \rightarrow \mathbb{R}_{\geq 0}\big|  v$ is continuous on $\mathbb{R}_{\geq 0}\big\}$, $C(\mathbb{R}_{\geq 0};\mathbb{R}):=\big\{v: {\mathbb{R}_{\geq 0}} \rightarrow \mathbb{R}\big|  v$ is continuous on $\mathbb{R}_{\geq 0}\big\}$,  and $C^1 (\mathbb{R}_{>0};\mathbb{R}):=\big\{v: {\mathbb{R}_{> 0}} \rightarrow \mathbb{R}\big|  v$ has continuous derivatives up to order $1$ on $\mathbb{R}_{> 0}\big\}$.
	
	For a domain $\Omega$  (either open or closed) in $\mathbb{R}^N$, let $C\left({\Omega}\right):=\big\{v: {\Omega} \rightarrow \mathbb{R}\big|  v$ is continuous on $\Omega\big\}$ and
	$C^1\left({\Omega}\right):=\{v: {\Omega} \rightarrow\mathbb{R}\big|  v$ has continuous derivatives up to order $1$ on ${\Omega}$\big\}.
	Let $C(\Omega\times\mathbb{R}_{\geq 0};{\mathbb{R}_{>0}}):=\big\{v : {\Omega\times\mathbb{R}_{\geq 0}} \rightarrow {\mathbb{R}_{>0}}\big|  v$ is continuous on $\Omega\times\mathbb{R}_{\geq 0}\big\}$,  $C(\Omega\times\mathbb{R}_{\geq 0};{\mathbb{R}}):=\big\{v : {\Omega\times\mathbb{R}_{\geq 0}} \rightarrow {\mathbb{R}}\big|  v$ is continuous on $\Omega\times\mathbb{R}_{\geq 0}\big\}$, and $C(\Omega\times\mathbb{R}_{\geq 0}\times\mathbb{R};\mathbb{R}):=\big\{v : {\Omega\times\mathbb{R}_{\geq 0}\times\mathbb{R}} \rightarrow \mathbb{R}\big|  v$ is continuous on $\Omega\times\mathbb{R}_{\geq 0}\times\mathbb{R}\big\}$.
	For  a normed linear space, let $C(\mathbb{R}_{\geq 0},{ U}):=\big\{v: {\mathbb{R}_{\geq 0}} \rightarrow  U\big|  v$ is continuous on $\mathbb{R}_{\geq 0}\big\}$.
	For $p\in[1,\infty)$, the space $L^p(\Omega)$  consists of  measurable functions $g: \Omega\to\mathbb{R}$ satisfying $ \int_\Omega\left|g(\xi)\right|^p\mathrm{ d}\xi <+\infty$ with norm $\|g\|_p:=\left(\int_\Omega\left|g(\xi)\right|^p\mathrm{ d}\xi\right)^{\frac{1}{p}}$. Define the norm $\|\cdot\|$ as $\|\cdot\|:=\|\cdot\|_2$. The space $W^{1,p}(\Omega)$ consists of measurable functions $f: \Omega\to\mathbb{R}$ such that $f$ and its weak derivatives up to order $1$ are in $L^p(\Omega)$ with norm $\|f\|_{1,p} := \left( \int_\Omega \left|f(\xi)\right|^p \, \mathrm{d}\xi + {\sum_{i=1}^N} \int_\Omega \left|\frac{\partial f}{\partial \xi_i}(\xi)\right|^p \, \mathrm{d}\xi \right)^{\frac{1}{p}}$.

	For Banach spaces $X$ and $\mathbb{Y}$, let ${\mathscr{L}}( X ,\mathbb{Y})$ be  the space of  bounded linear operators $S$ from $ X $ to $\mathbb{Y}$   with the  norm $\|S\|:=\sup\left\{\|Sx\|:x\in X ,\|x\| _{ X }\leq1\right\}$, and ${\mathscr{L}}( X ):={\mathscr{L}}( X , X )$.
	For an operator {$B$,   $D(B)$ denotes the domain of $B$ and  $B^{*} $ denotes the adjoint of $B$. For $\phi\in (0,\pi)$, let  $S_\phi:=\left\{z\in \mathbb{C}\backslash \{0\}:|\arg z|\leq \phi\right\}$, where $\mathbb{C}:=\big\{a+b\text{i}~\big|~a,b\in \mathbb{R}, \text{i}^2=-1\big\}$.}
\section{Comparison principles for nonlinear ODEs with time-varying coefficients}\label{comparison principle}

In this section, we prove comparison principles for a wide range of nonlinear ODEs with time-varying coefficients and superlinear terms, which provide crucial tools for establishing the LiISS of non-autonomous infinite-dimensional systems within the framework of Lyapunov stability theory.

Let's first introduce some auxiliary results. The following lemma is    a generalization of \cite[Lemma 4.4]{Lin1996a}.

\begin{lemma}\label{lemma 2.1}
	Assume that $\alpha,g  \in C(\mathbb{R}_{\geq 0};\mathbb{R}_{\geq 0})$  with $\lim\limits_{s\rightarrow +\infty}\int_{0}^sg(\tau) {\rm{d}}\tau=+\infty$.
	For $T\in \mathbb{R}_{>0}$ and $y_0\in \mathbb{R}_{\geq 0}$, if $y$ is   nonnegative and absolutely continuous   on $[0,T]$  and satisfies the following differential inequality:
	\begin{equation}\label{ODE}
		\left\{
		\begin{array}{ll}
			y'(t)\leq -g(t)\alpha(y(t))  \  \text{a.e. in}\  [0,T],\\
			y(0)=y_0,
		\end{array}\\
		\right.
	\end{equation}
	then there is a function $\beta \in\mathcal{K}\mathcal{L}$ such that
	\begin{align*}
		y(t)\leq \beta (y_0,t), \forall t\in [0,T].
	\end{align*}
	Particularly, when $ \alpha(y)=y$ for all $y\in \mathbb{R}_{\geq 0}$,
	the function $\beta$ is given by
	\begin{align*}
		\beta (s,t)=se^{-\int_{0}^tg (\tau){\rm d}\tau},\forall s\geq 0,t\in[0,T].
	\end{align*}
\end{lemma}
\begin{proof}
	Let $
	\overline{\alpha}(s):=\min\{s,\alpha(s)\}
	$, which is nonnegative, continuous,  satisfies $\overline{\alpha}(s)\leq \alpha(s)$ for all $s\geq 0$, and
	\begin{align*}
		\lim\limits_{s\rightarrow 0^+}\int_{s}^1\frac{1}{\overline{\alpha}(\tau)}{\rm d}\tau\geq \lim\limits_{s\rightarrow 0^+}\int_{s}^1\frac{1}{ \tau }{\rm d}\tau=+\infty.
	\end{align*}
	For any $s>0$, let $\eta(s):=\int_{s}^1\frac{1}{\overline{\alpha} (\tau)}{\rm d}\tau$. It is clearly that $\lim\limits_{s\rightarrow 0^+}\eta(s)=+\infty$  and $\eta(s)$ is strictly decreasing on $[0,+\infty)$. Denote by $\eta^{-1}(s)$ the inverse of $\eta(s)$.
	
	For any $t\in [0,T]$, let $G(t):= \int_{0}^tg (\tau){\rm d}\tau$ and  \begin{equation}\label{beta}
		\beta(s,t):=\left\{
		\begin{array}{ll}
			0,& s=0,\\
			\eta^{-1}(\eta(s)+G(t)),&s>0.
		\end{array}\\
		\right.
	\end{equation}
	Note that $G(0)=0$ and  $ \beta (s,0)=s$ for all $s\geq0$. We claim that the function $  \beta$ defined by \eqref{beta} is a $\mathcal{K}\mathcal{L}$-function. Indeed, by the continuity of $\eta(s)$, $\eta^{-1}(s)$, and $G(t)$, and noting that $ \lim\limits_{s\rightarrow +\infty}\eta^{-1}(s)=0$,  we see that $  \beta(s,t)$  is continuous in both $s$ and $t$. Due to the strictly decreasing property of $\eta(s)$ and  $\eta^{-1}(s)$, $  \beta(s,t)$ is strictly increasing in $s$ for any fixed $t$. In view of the condition of $g$ and~\eqref{beta}, it follows that $\lim\limits_{t\rightarrow +\infty}\beta(s,t)=0$ for any fixed $s$.  We deduce that $  \beta$ defined by \eqref{beta} is a $\mathcal{K}\mathcal{L}$-function.

	We infer from \eqref{ODE} and the definition of $\overline{\alpha}(s)$ that
	\begin{align*}
		\int^{y(t)}_{y_0}\frac{1}{\overline{\alpha} (\tau)}{\rm d}\tau= \int^{t}_{0}\frac{y'(\tau)}{\overline{\alpha} (y(\tau))}{\rm d}\tau
		\leq - \int_0^t\frac{g (\tau) {\alpha} (\tau)} {\overline{\alpha} (\tau)} {\rm d}\tau
		\leq  - \int_0^t g (\tau) {\rm d}\tau
		= -G(t),\forall t\in[0,T].
	\end{align*}
	It follows that
	\begin{equation*}
		\int_{y(t)}^{1}\frac{1}{\overline{\alpha} (\tau)}{\rm d}\tau\geq \int_{y_0}^{1}\frac{1}{\overline{\alpha} (\tau)}{\rm d}\tau+G(t),\forall t\in [0,T],
	\end{equation*}
	which is equivalent to
	\begin{equation*}
		\eta(y(t))\geq \eta(y_0)+G(t),\forall t\in[0,T].
	\end{equation*}
	By the definition of $\eta(s)$, and noting that  $\eta^{-1}(s)$ is decreasing in $ s\geq 0$, we have
	\begin{equation*}
		y(t)=\eta^{-1}(\eta(y(t)))\leq \eta^{-1}(\eta(y_0)+G(t))= \beta (y_0,t),\forall t\in[0,T],
	\end{equation*}
	which, along with the continuity of $y(t)$ and $\beta(s,t)$, gives
	\begin{equation*}
		y(t)=\eta^{-1}(\eta(y(t)))\leq \eta^{-1}(\eta(y_0)+G(t))= \beta (y_0,t),\forall t\in [0,T] .
	\end{equation*}
	In particular, if $ \alpha(y)=y$,
	then $\eta(s)=\int_{s}^1\frac{1}{\alpha (\tau)}{\rm d}\tau=-\ln s$, $\eta^{-1}(s)=e^{-s}$.
	It follows from \eqref{beta} that
	\begin{equation*} \beta (s,t)=se^{-\int_{0}^tg (\tau){\rm d}\tau},\forall s\geq 0,t\in[0,T].\end{equation*}
\end{proof}

The next lemma can be seen as a generalization of \cite[Lemma IV.2]{angeli2000IEEETAC}.

\begin{lemma}\label{lemma 2.2}Assume that $\alpha,g  \in C(\mathbb{R}_{\geq 0};\mathbb{R}_{\geq 0})$  with $\lim\limits_{s\rightarrow +\infty}\int_{0}^sg(\tau) {\rm{d}}\tau=+\infty$.
	For $T\in \mathbb{R}_{>0}$ and $y_0\in \mathbb{R}_{\geq 0}$, if $y$ is   nonnegative and absolutely continuous   on $[0,T]$  and satisfies the following differential inequality:
	\begin{equation*}
		\left\{
		\begin{array}{ll}
			y'(t)\leq -g(t)\alpha\left(\max\{y(t)+v(t),0\}\right)  \ \text{a.e.  in}\  [0,T],\\
			y(0)=y_0 ,
		\end{array}\\
		\right.
	\end{equation*}
	with $v\in   C([0,T];\mathbb{R}_{\geq 0})$, then there is a function $\beta \in\mathcal{K}\mathcal{L}$ such that
	\begin{align*}
		y(t)\leq \max\left\{\beta (y_0,t),\max_{s\in[0,t]}  v(s)  \right\}, \forall t\in [0,T].
	\end{align*}
\end{lemma}
\begin{proof} By virtue of Lemma \ref{lemma 2.1} and $g \geq 0$ on $[0,+\infty)$, one may proceed   in the  same way as in the proof of \cite[Lemma IV.2]{angeli2000IEEETAC} to obtain the desired result.
\end{proof}

Using Lemma \ref{lemma 2.2}, we can prove the following result, which is a generalization of  \cite[Corollary IV.3]{angeli2000IEEETAC}.
\begin{lemma}\label{lemma 2.3}Assume that $\alpha \in C^1(\mathbb{R}_{> 0};\mathbb{R})\cap C (\mathbb{R}_{\geq 0};\mathbb{R}_{\geq 0})$ and $ g \in C(\mathbb{R}_{\geq 0};\mathbb{R}_{\geq 0})$  with $\lim\limits_{s\rightarrow +\infty}\int_{0}^sg(\tau) {\rm{d}}\tau=+\infty$.
	For $T\in \mathbb{R}_{>0}$ and $y_0\in \mathbb{R}_{\geq 0}$, if $y$ is   nonnegative and absolutely continuous   on $[0,T]$  and satisfies the following differential inequality:
	\begin{equation*}
		\left\{
		\begin{array}{ll}
			y'(t)\leq -g(t)\alpha(y(t))+v(t) \   \text{a.e.\  in}\  [0,T],\\
			y(0)=y_0,
		\end{array}\\
		\right.
	\end{equation*}
	with $v\in   C([0,T];\mathbb{R}_{\geq 0})$, then there is a function $\beta \in\mathcal{K}\mathcal{L}$, which is determined by Lemma \ref{lemma 2.2}, such that
	\begin{align*}
		y(t)\leq \beta (y_0,t)+2\int_{0}^tv(s){\rm d}s, \forall t\in [0,T].
	\end{align*}
\end{lemma}
\begin{proof}
	Since $\alpha \in C^1 (\mathbb{R}_{>0};\mathbb{R})$,  $\alpha$ is locally Lipschitz continuous  on $[0,+\infty)$. Therefore, the   comparison principle,  e.g., \cite[Lemma 3.4]{Khalil2001non}, can be applied  to the differential equations having a form $z'(t)= -g(t)\alpha(z(t))+v(t)$.  Recall  that $g \geq 0$ on $[0,+\infty)$. One may proceed in the  same way as in the proof of  \cite[Corollary IV.3]{angeli2000IEEETAC} to obtain the desired result.
\end{proof}

Now, we prove the following result, which is crucial for the LiISS analysis of the non-autonomous system \eqref{systems}.
\begin{proposition}\label{lemma 2.4}
	Assume that the following two conditions are satisfied:
	\begin{enumerate}
		\item[(i)] $g_1$ and $g_2$ are nonnegative and locally integrable on $[0,+\infty)$ and satisfy  \begin{align*}
			\lim\limits_{t\rightarrow +\infty}\int_0^tg_1(\tau){\rm{d}}\tau=+\infty \quad \text{and} \quad
			g_2(t)\leq g_1(t),\forall t\geq 0;
		\end{align*}
		\item[(ii)]   $\alpha_{3}  ,\alpha_4\in C(\mathbb{R}_{\geq 0};\mathbb{R})$, and there exists a  constant ${r'}>0$ such that
		$\alpha_3-\alpha_4\in  C^1((0,r'];\mathbb{R}) \cap C ([0,r'];\mathbb{R}_{\geq 0})$.
	\end{enumerate}
	For $T\in \mathbb{R}_{>0}$,  $y_0\in \mathbb{R}_{\geq 0}$, and $v\in   C([0,T];\mathbb{R}_{\geq 0})$ satisfying $y_0+2\int_0^{T}v(t){\rm d}t< r'$, if $y$ is   nonnegative and absolutely continuous   on $[0,T]$  and satisfies the following differential inequality:
	\begin{equation}\label{ODE''}
		\left\{
		\begin{array}{ll}
			y'(t)\leq -g_1(t)\alpha_3(y(t))+g_2(t)\alpha_4(y(t))+v(t)  \  \text{a.e.\  in}\  [0,T],\\
			y(0)=y_0,
		\end{array}\\
		\right.
	\end{equation}
	then   there exists a function $\beta \in\mathcal{K}\mathcal{L}$   such that
	\begin{align}
		y(t)\leq \beta (y_0,t)+2\int_{0}^tv(s){\rm{d}}s, \forall t\in [0,T].\label{estimate-y}
	\end{align}
	Furthermore, if $g_2(t)\equiv 0$ for all $t\in \mathbb{R}_{\geq 0}$, then the estimate \eqref{estimate-y} holds  true for all $y_0\in \mathbb{R}_{\geq 0}$  and $v\in   C([0,T];\mathbb{R}_{\geq 0})$.
	
\end{proposition}
\begin{proof}   By condition~(ii),  there exists $r'>0$ such that
	\begin{equation}
		\label{alpha2}
		\alpha_4(y)\leq  \alpha_3(y),\ \forall y\in [0,r'].
	\end{equation}
	
	Let  $t^*:=\inf\big\{t\in[0,T]\big|  y(t)> r' \big\}$, or $t^*:=T$ for  $\big\{t\in[0,T]\big| y(t)> r' \big\}=\emptyset$.  Note that $y(0)=y_0< r'$. Thus $t^*\neq 0$. Therefore, $t^*\in(0,T)$ or $t^*=T$.
	
	Now, we claim that $t^*=T$. If not, then $t^*\in(0,T)$.  Using the definition of $t^*$, we have
	\begin{equation}
		\label{alpha2'}
		y(t)\leq r',\forall t\in[0,t^*).
	\end{equation}
	We infer from \eqref{ODE''}, \eqref{alpha2}, \eqref{alpha2'}, and condition~(i) that
	\begin{equation*}
		y'(t)\leq - g_1(t)(\alpha_3(y(t))-\alpha_4(y(t)))+v(t)\   \text{a.e.\  in}\   [0,t^*).
	\end{equation*}
	
	According to Lemma \ref{lemma 2.3}, there exists $\beta  \in\mathcal{K}\mathcal{L}$ such that
	\begin{equation}
		\label{17}
		y(t)\leq \beta (y_0,t)+2\int_{0}^tv(s){\rm d}s , \forall t\in  [0,t^*),
	\end{equation}
	which, along with the decreasing property of $ \beta (s,t)$ in $t$, implies
	\begin{align*}
		r'   =&y(t^*)\notag\\
		=&\lim\limits_{t\rightarrow (t^*)^-}y(t) \notag\\
		\leq& \lim\limits_{t\rightarrow (t^*)^-} \left( \beta  (y_0,t)+2\int_{0}^tv(s){\rm d}s\right) \notag\\
		=&   \beta (y_0,t^*)+2\int_{0}^{t^*}v(s){\rm d}s
		\notag\\
		\leq& \beta (y_0,0)+2\int_{0}^{t^*}v(s){\rm d}s \notag\\
		=& y_0+2\int_{0}^{t^*}v(s){\rm d}s\\
		<&r'.
	\end{align*}
	We get a contradiction! Thus $t^*=T$.
	
	Furthermore, proceeding as above, we may prove that the inequality \eqref{17} holds true for all $t\in [0,T]$, provided that $y_0 + 2\int_0^{T}v(t){\rm d}t< r'$.
\end{proof}

\section{{LiISS} of non-autonomous infinite-dimensional systems}\label{Section-2}

Let $ X $  be  a Banach space with   norm $\|\cdot\|_{ X }$ and $ U\subset  X $ be a normed linear space with  norm {$\|\cdot\|_{U}$}, respectively. 
In this section, we consider the LiISS for a class of non-autonomous infinite-dimensional systems  govern by the following   evolution equation:
\begin{subequations}\label{systems}
	\begin{align}
		\dot{x}(t)= & A(t)x(t)+F(t, x(t), u(t)),  t\in\mathbb{R}_{\geq0}, \label{systems-a}\\
		x(0)= & x_0,  \label{systems-b}
	\end{align}
\end{subequations}
{where $x(t)\in  X $ is the state, $x_0$ represents the initial state, $u(t)\in  U$ is the control input,   $A(t):D(A(t))\subset  X  \to  X $ is a bounded linear operator for any fixed $t\in \mathbb{R}_{\geq 0}$, and $F$ is a nonlinear  function, which  particularly  is a  mapping from a {subspace} of $X$  to $X$ rather than from $X$ to $X$ when $t$ and $u$ are fixed.}

Assume that for any fixed $t\in \mathbb{R}_{\geq 0}$, $A(t)$   generates
an analytic semigroup  on the Banach space $X$. Then,   $D(A(t))$ is dense in $X$ for any   $t\in \mathbb{R}_{\geq 0}$ (see \cite[Corollary~1.2.5]{Pazy2012semigroups}). Furthermore, for any fixed $t\in \mathbb{R}_{\geq 0}$ and $\alpha\in [0,1]$,  one can define the fractional power operator
$(-A(t))^{-\alpha} :  X  \rightarrow  X $   via the contour integral
\begin{align*}(-A(t))^{-\alpha} = \int_{\partial S_{\phi'}} z^{-\alpha}(zI + A(t))^{-1} {\rm d}z,
\end{align*}
where $I$ is the identity operator and $\partial S_{\phi'}$ is the boundary of a  sufficiently large sector   $S_{\phi'}$, which particulary contains the spectrum of $ A(t) $ in the complex plane for any fixed $t\in \mathbb{R}_{\geq 0}$.  Since $(-A(t))^{-\alpha}$ is a bounded injective operator on $X$ for any fixed $t\in \mathbb{R}_{\geq 0}$, one can further define $(-A(t))^{\alpha}: = ((-A(t))^{-\alpha})^{-1}: \text{range}((-A(t))^{-\alpha})\to  X$. The domain of $(-A(t))^{\alpha}$ equipped with the graph norm, i.e., $|||x|||:=\|x\|_{X}+\|(-A(t))^{\alpha}x\|_X$, is a Banach space for any fixed $t\in \mathbb{R}_{\geq0}$. Since $(-A(t))^{\alpha}$ is invertible, its graph norm is equivalent to the  norm $\| (-A(t))^{\alpha}x\|_X$. Thus,  the domain of $(-A(t))^{\alpha}$ equipped with the norm $\| (-A(t))^{\alpha}x\|_X$ is  a Banach space which we denote  by $X_{\alpha}$. For convenience, we denote $\|x\|_{X_{\alpha}}:=\| (-A(t))^{\alpha}x\|_X$.
From this definition it is clear that
$D(A(t))   \subset X _\beta  \subset X _\alpha \subset  X  $ for $0< \alpha<\beta < 1$ and  the imbedding of $X _\beta$ in $X _\alpha$ is
continuous;   see \cite[Sec. 2.6]{Pazy2012semigroups}.

{Throughout this paper, assume that {$F :\mathbb{R}_{>0}\times  X _{\alpha}\times  U\rightarrow  X $}, $x_0\in D(A(0)) $, and
	$u\in U_c:=C(\mathbb{R}_{\geq 0},{ U})$.}

In this section, for    system \eqref{systems},
we   prove an LiISS Lyapunov theorem in the framework of Banach spaces, then provide sufficient conditions to ensure the existence of an LiISS-LF and construct LiISS-LFs  in the framework of Hilbert spaces.

\subsection{The LiISS Lyapunov theorem   in the   framework of  Banach spaces}\label{sec:2}



	
	In this section, we prove an LiISS Lyapunov theorem for system \eqref{systems} in the framework of Banach spaces.
	It is worth mentioning that  a
	function $x:\mathbb{R}_{\geq0}\to X$ is called a  classical solution {to} system~\eqref{systems} if $x(t) $ is continuous in $t\in \mathbb{R}_{\geq0}$,  $x(t)$ is continuously differentiable in $t\in \mathbb{R}_{>0}$, $x(t)\in D(A(t))$ for any $t\in \mathbb{R}_{>0}$, and satisfies \eqref{systems}.
	Note that under appropriate assumptions on $A(t)$ and $F(t,x,u)$, the existence of a   classical solution (or a strong or mild solution)  has been proven in \cite{Pazy2012semigroups} for system \eqref{systems} in various cases, especially, in the disturbance-free case; {see also \cite{Heni:2025} for the well-posedness of  time-varying semi-linear evolution equations with locally Lipschitz continuous nonlinearities and piecewise-right continuous inputs}.
	Since the main purpose of the present work is to demonstrate how to perform LiISS analysis on non-autonomous infinite-dimensional systems with a wide range of forms within the framework of Lyapunov stability theory, rather than proving the existence and regularity of solutions,  throughout this paper, we  consider the LiISS  of system~\eqref{systems} by means of classical solutions.

First, we introduce  the concept of LiISS  for system \eqref{systems}.

\begin{definition}
	System \eqref{systems} is said to be integral input-to-state stable (iISS) if there {exist} functions $\gamma \in \mathcal{K}$, $\sigma\in\mathcal{K}_{\infty}$, $\beta\in \mathcal{K}\mathcal{L} $  such that, for any initial datum $x_{0}\in D(A(0)) $ and input ${u\in U_c}$, the solution {to} system \eqref{systems} satisfies the following inequality:
	\begin{align}\label{iISS}
		\|x(t, x_{0},  u)\|_{ X } \leq {\beta}(\|x_{0}\|_{ X } ,  t) + \sigma\left(\int^{t}_{0}\gamma(\|u(s)\|_{ U}){\rm{d}}s\right), \forall t\in\mathbb{R}_{\geq0}.
	\end{align}
	{In particular, if there {exist}  a constant  $R \in \mathbb{R}_ {>0}$ and functions $\gamma_0\in\mathcal{K}, \sigma_{0}\in\mathcal{K}_{\infty}$ such that the estimate \eqref{iISS}  holds true for any
		\begin{align*}({x}_0, u) \in
			\mathscr{B}_{R}: =\left\{ (x_{0},u) \in  D(A(0)) \times  U_c\left|  \|x_{0}\|_{ X }+\sigma_{0}\left(\int^{t}_{0}\gamma_0(\|u(s)\|_{ U}){\rm{d}}s\right) \leq R,\forall t\in \mathbb{R}_{>0}\right.\right\},
		\end{align*}
		then, system \eqref{systems} is said to be locally integral input-to-state stable (LiISS). }
\end{definition}

\begin{remark} {Note that $D(A(t))$ is dense in $X$ for all $t\in \mathbb{R}_{\geq 0}$.
		If the solution {to} system~\eqref{systems}   depend continuously on the inputs and the initial states, then, based on  a standard density argument (see, e.g., \cite{mironchenko2023arxiv}),   the iISS and LiISS of system \eqref{systems} can  also be defined via  ${x}_0\in X$ and mild solutions. }
	%
	
\end{remark}

Next, we   introduce the concept of  {LiISS Lyapunov} functional  for system \eqref{systems}. Analogous to the classical Lyapunov theory for stability analysis, the LiISS-LF is a crucial tool for verifying the LiISS property of the system.
\begin{definition}\label{DLiISS}
	A {continuously differentiable} functional $V:\mathbb{R}_{\geq 0}\times  X \to \mathbb{R}_{\geq 0}$ is called an LiISS Lyapunov functional {(LiISS-LF)} of system \eqref{systems}, if it has the following properties:
	\begin{enumerate}
		\item[(i)]  there {exist} two functions $ \alpha_{1},\alpha_{2} \in \mathcal{K}_{\infty}$  such that
		\begin{align}\label{V}
			{\alpha_{1}}(\|x\|_{ X })\leq {V(t, x )}\leq{\alpha_{2}}(\|x\|_{ X }),  \forall x\in X ,  t\in\mathbb{R}_{\geq0};
		\end{align}
		\item[(ii)]  there {exist two nonnegative} locally integrable functions $g_1, g_2$ defined on $\mathbb{R}_{\geq 0}$ satisfying \begin{align*}
			\lim\limits_{t\rightarrow +\infty}\int_0^tg_1(\tau){\rm{d}}\tau=+\infty\quad \text{and}\quad g_2(t)\leq g_1(t),\forall t\geq0,
		\end{align*}
		two functions $\theta_{1},\theta_{2} \in\mathcal{K}$  satisfying
		\begin{align}
			\theta_{2}(s)\leq\theta_{1}(s),  \forall s\in[0, R'] \label{4}
		\end{align}
		with {a constant $R'\in \mathbb{R}_{>0}$,}
		and a function $\phi\in\mathcal{K}$,   such that  the  derivative of $V$ along the trajectory of   system \eqref{systems}  satisfies the following inequality:
		\begin{align}\label{DV}
			{\dot{V}(t, x)} \leq -g_{1}(t)\theta_{1}(\|x\|_{ X })+g_{2}(t)\theta_{2}(\|x\|_{ X })+\phi(\|u(t)\|_{ U}) ,  \forall   t \in\mathbb{R}_{\geq0}.
		\end{align}
	\end{enumerate}
	
	{Furthermore, if $g_2(t)\equiv 0$ in the condition (ii), then $V$ is called an  iISS Lyapunov functional (iISS-LF) of system \eqref{systems}.}
\end{definition}

Now, we state      the LiISS (or iISS) Lyapunov  theorem for system \eqref{systems}.
\begin{theorem} \label{TLiISS}
	{If system \eqref{systems} admits an LiISS (or iISS) Lyapunov functional, then it is LiISS (or iISS).}
\end{theorem}

\begin{proof} 	
	Consider the      maximal time interval ${[0,T]}$ with some $T\in\mathbb{R}_{> 0}$ or $[0,+\infty)$, over which   system~\eqref{systems} admits a {classical} solution $x$ corresponding to the initial state and input  $(x_{0},u) \in  D(A(0)) \times U $. Let $V(t,x)$ be an LiISS-LF of system~\eqref{systems} and $y(t):=V(t, x(t))$.
	
	If the  maximal time interval for the existence of a solution is  ${[0,T]}$ with some $T\in\mathbb{R}_{> 0}$, then we deduce from \eqref{V} and \eqref{DV}   that the derivative of $y$ along the trajectory $x$ satisfies
	\begin{align}\label{18}
		\dot{y}(t)\leq-g_1(t)\theta_{1}\left({\alpha_{2}}^{-1}(y(t))\right)+g_2(t)\theta_{2}\left({\alpha_{1}}^{-1}(y(t))\right)+{\phi}\left({\|u(t)\|_{ U}}\right),\forall t\in[0,T].
	\end{align}
	
	Note that the condition~\eqref{4} implies that
	\begin{align*}
		\alpha_3(s):=\theta_{1}\left({\alpha_{2}}^{-1}(s)\right) \geq \alpha_4(s):=\theta_{2}\left({\alpha_{1}}^{-1}(s)\right),\forall s\in \left[0,r'\right],
	\end{align*}
	where $r':=\min\{\alpha_1(R'),\alpha_2(R')\}$.
	
	Note also that if $(x_{0},u) $ belongs to $  \mathscr{B}_{R}$ with a positive constant {$R:=\alpha_2^{-1}\left(\frac{r'}{2}\right)$} and functions $\gamma_0(s):=2\phi(s), {\sigma_{0}(s):= \alpha_2^{-1}(s)}$ for $s\in \mathbb{R}_{\geq 0}$, then the condition~\eqref{V} and $(x_{0},u)\in \mathscr{B}_{R}$ ensure  that
	\begin{align*}
		y_0+  \int^{T}_{0}2{\phi}(\|u(s)\|_{ U}){\rm d}s\leq&  \alpha_2(\|x_0\|_X )+\int^{T}_{0}2{\phi}(\|u(s)\|_{ U}){\rm d}s\notag\\
		=&\alpha_2(\|x_0\|_X )+\alpha_2\left(\sigma_{0}\left(\int^{T}_{0}\gamma_0(\|u(s)\|_{ U}){\rm d}s\right)\right)\notag\\
		\leq &2\alpha_2\left(\|x_0\|_X +\sigma_{0}\left(\int^{T}_{0}\gamma_0(\|u(s)\|_{ U}){\rm d}s\right)\right)\notag\\
		\leq & 2\alpha_2( R ) \notag\\
		=&r'.
	\end{align*}
	Applying {Proposition}~\ref{lemma 2.4} to \eqref{18}, we deduce that
	there exists a function $\beta_{1}\in\mathcal{K}\mathcal{L}$   such  that
	\begin{align*}
		y(t)\leq\beta_{1}(y(0), t)+2\int^{t}_{0}{\phi}(\|u(s)\|_{ U}){\rm d}s,\forall t\in[0,T].
	\end{align*}
	In view of \eqref{V}, it holds that
	\begin{align}\label{19}
		\alpha_{1}(\|x\|_{ X })\leq \beta_{2}(\|x_{0}\|_{ X }, t)+2\int^{t}_{0}{\phi}(\|u(s)\|_{ U}){\rm d}s,\forall t\in[0,T],
	\end{align}
	where $\beta_{2}(s,r):=\beta_{1}({\alpha_{2}}(s),  r)$ for $ s,r \in   \mathbb{R}_{\geq 0} $ is  a $ \mathcal{K}\mathcal{L}$-function.
	
	Let $\sigma(s) := \alpha_{1}^{-1}(2s),  {\gamma}(s) := 2{\phi}(s)$, and $\beta(s, r) =\alpha_{1}^{-1}(2\beta_{2}(s, r))$ for $ s,r \in  \mathbb{R}_{\geq 0} $. Then  \eqref{19} implies that the inequality \eqref{iISS} holds true for all $t\in[0,T]$.   Meanwhile, we deduce that the solution {to} system~\eqref{systems} is uniformly bounded. Therefore, using \cite[Theorem 1.4]{Pazy2012semigroups}, we infer that the   maximal time interval must be $[0,+\infty)$.
	Consequently,   the inequality \eqref{iISS} holds true for all $t\in\mathbb{R}_{\geq0}$. We conclude that system~\eqref{systems} is LiISS.
	
	Furthermore,   if system~\eqref{systems} admits an iISS-LF, then, by virtue of  Lemma~\ref{lemma 2.4} with $g_2(t)\equiv0$, it is easy to see that \eqref{19} holds {true} for all $(x_{0},u) \in  D(A(0)) \times U $. Thus, system~\eqref{systems} is iISS.
\end{proof}

\subsection{Constructing LiISS-LFs   in the framework of Hilbert spaces}\label{sec:3}
In this section, let $ X $  be  a Hilbert space  with   scalar inner product $\langle \cdot,\cdot \rangle$ and   norm $\| \cdot \|_ X :=\sqrt{\langle \cdot,\cdot \rangle}$. 
We provide sufficient conditions to ensure the existence of an LiISS-LF and construct LiISS-LFs for system \eqref{systems}  in the framework of Hilbert spaces. More specifically, we prove the following result.

\begin{theorem}\label{Thm2}
Suppose that there exist  a continuously differentiable,  {uniformly bounded,    uniformly  coercive}, and  positive-definite operator $P(t)\in {\mathscr{L}}( X )$  defined for any $t\in\mathbb{R}_{\geq0}$,  nonnegative and locally integrable functions $a_1(t), a_2(t)$ defined over $[0,+\infty)$ and satisfying
\begin{align*}
	\lim\limits_{t\rightarrow +\infty}\int_0^ta_1(\tau){\rm{d}}\tau=+\infty\quad   \text{and}\quad
	a_2(t)\leq M a_1(t)
\end{align*}
with a constant $M\geq 0$,   constants  $m_1,m_2>2$,  and a function $\zeta \in\mathcal{K}$, such that the following inequality holds true:
\begin{align}\label{ro}
	\langle ( {A^*(t)P(t)+P(t)A(t)+\dot{P}(t)})x , x\rangle+2\langle F(t, x, u), P(t)x \rangle
	\leq&-a_1(t)\|x\|^2_{ X }+ a_2(t) {\left(\|x\|^{m_1}_ X +\|x\|^{m_2}_ X \right)}\notag\\
	&+ {\zeta({\|u\|_{U}})}, \forall x\in D(A),{u\in U}, t\in\mathbb{R}_{\geq 0}.
\end{align}
Then, $ V(t,x):=\langle {P(t)}x, x \rangle $ is {an} {LiISS-LF}, and therefore, system \eqref{systems} is LiISS. In particular, if $a_2\equiv 0$ on $[0,+\infty)$, then $ V(t,x):=\langle {P(t)}x, x \rangle $ is {an} {iISS-LF}, and hence, system \eqref{systems} is iISS.
\end{theorem}

\begin{proof}
{Since $P(t)$ is  a uniformly   bounded, uniformly coercive, and  positive-define operator,} there exist  $\eta_1,\eta_2\in \mathbb{R}_{>0}$ such that
\begin{equation*}
	\eta_1  \|x\|^2_{ X }\leq \langle P(t)x, x \rangle\leq  \|P(t)\| \|x\|^2_{ X }\leq \eta_{2} \|x\|^2_{ X },  \forall x\in  X ,
\end{equation*}
namely,
\begin{align*}
	\eta_{1}\|x\|^{2}_{ X }\leq V(t,x)\leq\eta_{2}\|x\|^{2}_{ X }, \forall x\in X .
\end{align*}

The derivative of $V$ along the trajectory of system~\eqref{systems} is given as follows:
\begin{align*}
	\dot{V}_{u}(t,x)= & {\langle \dot{P}x, x\rangle}+\langle P\dot{x}, x\rangle+\langle Px, \dot{x}\rangle\\
	= & {\langle \dot{P}x, x\rangle}+\langle P(Ax+F(t, x, u)), x\rangle+\langle Px, Ax+F(t, x, u)\rangle\\
	= & {\langle \dot{P}x, x\rangle}+\langle PAx, x\rangle+\langle PF(t, x, u), x\rangle+\langle Px, Ax\rangle+\langle Px, F(t, x, u)\rangle\\
	= & {\langle \dot{P}x, x\rangle}+\langle (PA+A^{*}P)x, x\rangle+\langle PF(t, x, u), x\rangle+\langle Px, F(t, x, u)\rangle\\
	= & {\langle (\dot{P}+PA+A^{*}P)x, x\rangle +2\langle Px, F(t, x, u)\rangle},
\end{align*}
which, along with \eqref{ro}, yields
\begin{align*}
	\dot{V}_{u}(t, x)
	\leq&{-a_1(t)\|x\|^2_{X}+ a_2(t)\left(\|x\|^{m_1}_ X +\|x\|^{m_2}_ X \right)+ \zeta(\|u\|_{U})}\notag\\
	\leq&  {-a_1(t)\|x\|^2_{X}+  {a_1(t) M}\left(\|x\|^{m_1}_ X +\|x\|^{m_2}_ X \right)+ \zeta(\|u\|_{U})}.
\end{align*}

Let   $g_1(t)= g_2(t):= a_1(t)$, {$\theta_{1}(s):=s^2, \theta_{2}(s):=M\left(s^{m_1}+s^{m_2}\right),\phi(s):= \zeta(s)$  for $s\in \mathbb{R}_{\geq 0}$}. {Note that $m_1>2$ and $m_2>2$}. It is clear that $\theta_{2}(s)\leq \theta_{1}(s)$ for sufficiently small $s>0$.
According to   Definition~\ref{DLiISS}, we deduce that $V(t,x)= \langle P(t)x,  x\rangle$ is {an LiISS-LF} for system~\eqref{systems}. Then, by Theorem~\ref{TLiISS}, we infer that system~\eqref{systems} is LiISS. Furthermore, if $a_2(t)\equiv 0$, we conclude that $V(t,x) = \langle P(t)x,  x\rangle$ is {an  iISS-LF} for system~\eqref{systems}, and hence, system~\eqref{systems} is iISS.
\end{proof}

As a consequence of Theorem \ref{Thm2}, we  can prove the following result, which provides a sufficient condition to ensure the existence of an LiISS-LF for a class of autonomous infinite-dimensional systems with superlinear terms.
\begin{corollary}
Let  $A$  be time-invariant. Suppose that there exist
a   coercive positive-definite operator $P\in {\mathscr{L}}( X )$, constants $\alpha\in[0,1)$ and $\mu_i,\widetilde{m}_i,\widetilde{n}_i\in\mathbb{R}_{\geq 0}$ satisfying
\begin{align*}{ \widetilde{m}_i>2-\widetilde{n}_i}>0 ,i=1,2 ,
\end{align*}
and a function $\zeta \in\mathcal{K}$,   such that
\begin{align*}
	\langle ( {A^*P+PA})x , x\rangle\leq & -\|x\|^2_{ X _\alpha} ,  {\forall  x\in D(A),}
\end{align*}
and
\begin{align*}
	{\langle F(t, x, u),  Px \rangle}
	\leq& {\mu_1\|x\|_{ X }^{\widetilde{m}_1}\|x\|^{\widetilde{n}_1}_{ X _{\alpha}}+\mu_2\|x\|_{ X}^{\widetilde{m}_2}}\|x\|^{\widetilde{n}_2}_{ X _{\alpha}}+{\zeta(\|u\|_{U})} ,  \forall x\in D(A),{u\in U}.
\end{align*}
Then, $ V(x):=\langle {P}x, x \rangle $ is {an} {LiISS-LF}, and therefore, system \eqref{systems} is LiISS.
\end{corollary}
\begin{proof}
Using the Young's inequality, we have
\begin{align*}
	\|x\|_{ X }^{\widetilde{m}_i}\|x\|^{\widetilde{n}_i}_{ X _{\alpha}}
	\leq& \varepsilon\|x\|^2_{ X_{\alpha}}+\varepsilon^{-\frac{\widetilde{n}_i}{2-\widetilde{n}_i}}\|x\|^{\frac{2\widetilde{m}_i}{2-\widetilde{n}_i}}_ X , i=1,2,
\end{align*}
where $\varepsilon>0$ is a positive constant to be chosen later.
Then, we have
\begin{align*}
	\langle ( {A^*P+PA})x , x\rangle+2\langle F(t, x, u),  Px \rangle
	\leq& -\left(1-2\mu_1\varepsilon-2\mu_2\varepsilon\right)\|x\|^2_{X_{\alpha}}+2\mu_1\varepsilon^{-\frac{\widetilde{n}_1}{2-\widetilde{n}_1}}\|x\|^{\frac{2\widetilde{m}_1}{2-\widetilde{n}_1}}_ X +2\mu_2\varepsilon^{-\frac{\widetilde{n}_2}{2-\widetilde{n}_2}}\|x\|^{\frac{2\widetilde{m}_2}{2-\widetilde{n}_2}}_ X\notag\\
	& +2{\zeta(\|u\|_{U})}.
\end{align*}

Choosing $\varepsilon\in(0,1)$ to be sufficiently small   so that $1-2\mu_1\varepsilon-2\mu_2\varepsilon>0$ and using the fact {that} $\|x\|_{ X }\leq C\|x\|_{ X_{\alpha}}$ with some positive constant $C$, we deduce that
\begin{align*}
	\langle ( {A^*P+PA})x , x\rangle+2\langle F(t, x, u),  Px \rangle
	\leq&  -\frac{1}{C}\left(1-2\mu_1\varepsilon-2\mu_2\varepsilon\right)\|x\|^2_{ X }+2\mu_1\varepsilon^{-\frac{\widetilde{n}_1}{2-\widetilde{n}_1}}\|x\|^{\frac{2\widetilde{m}_1}{2-\widetilde{n}_1}}_ X +2\mu_2\varepsilon^{-\frac{\widetilde{n}_2}{2-\widetilde{n}_2}}\|x\|^{\frac{2\widetilde{m}_2}{2-\widetilde{n}_2}}_ X\notag\\
	& +2{\zeta(\|u\|_{U})}\notag\\
	\leq& -\frac{1}{C}\left(1-2\mu_1\varepsilon-2\mu_2\varepsilon\right)\|x\|^2_{ X }+\widetilde b\left(\|x\|^{\frac{2\widetilde{m}_1}{2-\widetilde{n}_1}}_ X +\|x\|^{\frac{2\widetilde{m}_2}{2-\widetilde{n}_2}}_ X \right)+2{\zeta(\|u\|_{U})},
\end{align*}
where $\widetilde b:=2\max\left\{\mu_1\varepsilon^{-\frac{\widetilde{n}_1}{2-\widetilde{n}_1}},\mu_2\varepsilon^{-\frac{\widetilde{n}_2}{2-\widetilde{n}_2}}\right\}$. Subsequently, the condition \eqref{ro} is fulfilled.  Theorem \ref{Thm2} ensures that
{$ V(t,x):=\langle {P(t)}x, x \rangle $} is {an} {LiISS-LF}. Therefore, {by Theorem \ref{TLiISS},}  system \eqref{systems} is LiISS.
\end{proof}

\begin{remark}
It is worth mentioning that for general nonlinear non-autonomous infinite-dimensional systems under an abstract form, verifying the  structural conditions of nonlinear terms is extremely challenging. As shown in Section \ref{sec:5.2}, even for specific PDEs,  it remains difficult to verify  the structural conditions in Theorem \ref{Thm2} when  superlinear terms are involved, usually requiring complex     techniques or  tools such as Sobolev embedding  and interpolation inequalities  to handle the nonlinear terms.
\end{remark}
%
%
%
%
%
%

\section{Illustrative examples}\label{sec:5}

In this section, we present two examples to illustrate the application of the LiISS Lyapunov theorem to an ODE system and a PDE
system, respectively, and provide numerical simulations to validate the LiISS results for the two systems.

\subsection{Local stabilization of
a nonlinear ODE system under linear state feedback control}\label{sec:5.1}
In this section, as an application of the LiISS Lyapunov theorem, we  investigate  the LiISS of  a  nonlinear ODE  system with time-varying coefficients and superlinear terms.   More {precisely}, we consider the linear state feedback control problem for the following nonlinear system
\begin{subequations}\label{nonlinear ODE}
\begin{align}
	\dot{x}_1 =&g(t)x^m_1-h_1(t)x_2^3 +u_1  , \\
	\dot{x}_2 =&{\tilde{g}(t)x^n_2}-b(t)x_1^2x_2+u_2,  \label{x2}
\end{align}
\end{subequations}
where $g\in C(\mathbb{R}_{\geq 0};\mathbb{R}_{\geq 0})$, ${\tilde{g}\in C(\mathbb{R}_{\geq 0};\mathbb{R}_{\geq 0})}$,   $b\in C(\mathbb{R}_{\geq 0};\mathbb{R}_{>0}) $,  $h_1\in C^1(\mathbb{R}_{\geq 0};\mathbb{R}_{\geq 0}) $ satisfying
\begin{align}\label{condition-h1}
0\leq h_1(t)\leq M\quad \text{and} \quad h_1'(t)\leq h_1(t),\quad \forall t\in \mathbb{R}_{\geq 0},
\end{align}
$M>0$ is a positive constant,    {$m>3$ and $n>1$} are  integers, and $u_1,u_2 $ are control inputs.

Note that the open-loop system~\eqref{nonlinear ODE} is not asymptotically stable at the origin for any given   initial data $x_1(0)$ and $x_2(0) $.
%
%
To apply a linear state-feedback controller,  we propose the following  control law using time-varying-gains:
\begin{subequations}\label{control law}
\begin{align}
	u_1(t):=&-g_1(t)x_1, \\
	u_2(t):=&h_2(t)x_1-g_2(t)x_2,
\end{align}
\end{subequations}
where $g_1,g_2,h_2\in C(\mathbb{R}_{\geq 0};\mathbb{R}_{\geq 0})$ satisfying
\begin{align}
g_1(t)\geq& {\overline{g}(t):=\max\left\{g(t),\tilde{g}(t)\right\}}, \forall t\in\mathbb{R}_{\geq0} \quad\text{and}\quad
\lim\limits_{t\rightarrow +\infty}\int_0^tg_1(\tau){\rm{d}}\tau=+\infty,\label{con-1}\\
g_2(t)\geq& 1+g_1(t),\forall t\in\mathbb{R}_{\geq0},\label{con-2}\\
h_2(t)=&\frac{h_1(t)}{1+h_1(t)},\forall t\in\mathbb{R}_{\geq0}\label{con-3}.
\end{align}
{Note that the coefficient matrix  of the linear part of system~\eqref{nonlinear ODE} under the control law \eqref{control law} is given by $A(t):=\begin{bmatrix}
-g_1(t) & 0\\
h_2(t) & -g_2(t)
\end{bmatrix}$, for which $g_1(t)$ may tend to $0$ as $t\to +\infty$.}
Note also that due to the presence of the term $g(t)x^m_1$ with a nonnegative function $g$ and positive integer $m$, the linear state-feedback control law \eqref{control law} can only
guarantee local asymptotic stability of system~\eqref{nonlinear ODE}. Furthermore, considering an additional disturbance $d(t)$ entering through $u_2$, the equation of   system~\eqref{nonlinear ODE} in closed loop becomes
\begin{subequations}\label{nonlinear ODE-closed loop}
\begin{align}
	\dot{x}_1 =&-g_1(t)x_1+g(t)x^m_1-h_1(t)x_2^3   , \\
	\dot{x}_2 =&h_2(t)x_1-g_2(t)x_2+{\tilde{g}(t)x^n_2}-b(t)x_1^2x_2+d(t) ,
\end{align}
\end{subequations}

 Assume that  {$d\in  L^{4} (0,T)$} for any $T>0$.  We have the following LiISS result  for system \eqref{nonlinear ODE-closed loop}.

\begin{proposition} \label{Prop.ODE}
System~\eqref{nonlinear ODE-closed loop} is LiISS.
\end{proposition}

To prove Proposition~\ref{Prop.ODE}, we need the following lemma.
\begin{lemma}\label{le-5}
For any $x:={(x_1,x_2)^{\top}}\in \mathbb{R}^2$, it holds that
$x_1^2+x_2^4\geq\frac{1}{2}\min {\left\{|x|^2, |x|^4\right\}}$.

\end{lemma}
\begin{proof}
On the one hand, for $x_2^2 \geq \frac{1}{2}$, we have
\begin{align*}
	x_1^2+x_2^4=x_1^2+\left(\frac{1}{2} \cdot \frac{x_2^2}{\frac{1}{2}}\right)^2 = x_1^2+\frac{1}{4}\left(\frac{x_2^2}{\frac{1}{2}}\right)^2\geq x_1^2+\frac{1}{4}\cdot\frac{x_2^2}{\frac{1}{2}} =x_1^2+\frac{1}{2} x_2^2 \geq \frac{1}{2}{|x|^2}.
\end{align*}
On the other hand, for $x_2^2 \leq \frac{1}{2}$ and  $x_1^2 \geq  \frac{1}{2}$, we  deduce that
\begin{align*}
	x_1^2+x_2^4 \geq 	x_1^2=\frac{1}{2} x_1^2+\frac{1}{2} x_1^2\geq \frac{1}{2} x_1^2+\frac{1}{2} x_2^2\geq \frac{1}{2}{|x|^2}.
\end{align*}
In addition, for $x_2^2 \leq  \frac{1}{2}$ and  $x_1^2 \leq  \frac{1}{2}$, we have
\begin{align*}
	x_1^2+x_2^4 \geq  x_1^4+x_2^4 \geq  \frac{1}{2}{|x|^4}.
\end{align*}
Thus, for any $x:=(x_1,x_2)^{\top}\in \mathbb{R}^2$, it holds that
$x_1^2+x_2^4\geq\frac{1}{2}\min {\left\{|x|^2, |x|^4\right\}}$.
\end{proof}

Now we prove the LiISS of system \eqref{nonlinear ODE-closed loop} by using the LiISS Lyapunov theorem. The key step is   constructing an appropriate LiISS Lyapunov candidate and verifying the structural conditions in Definition \eqref{DLiISS}.

\begin{proof of Prop.ODE}
For any $x:=(x_1,x_2)^{\top}\in \mathbb{R}^2$, define $V(t,x):=\frac{1}{2}x_1^2+\frac{1}{4}(1+h_1(t))x_2^4$. We claim that $V(t,x)$ is  an LiISS Lyapunov function of system~\eqref{nonlinear ODE-closed loop}.
Indeed, by \eqref{condition-h1},  we have
\begin{align}\label{s-e-v}
	V(t,x)&\leq   \frac{1}{2}x_1^2+\frac{1}{4}(1+M)x_2^4\notag\\
	&\leq  \frac{1}{2}\left(x_1^2+x_2^2\right)+\frac{1}{4}(1+M)\left(x_1^2+x_2^2\right)^2\notag\\
	&\leq \max\left\{ \frac{1}{2},\frac{1}{4}(1+M)\right\}{\left(|x|^2+|x|^4\right)},\forall x\in \mathbb{R}^2.
\end{align}
Moreover, according to \eqref{condition-h1} and Lemma \ref{le-5}, we have
\begin{align*}
V(t,x)\geq\frac{1}{2}x_1^2+\frac{1}{4} x_2^4\geq \frac{1}{4}\left(x_1^2+x_2^4\right) \geq  \frac{1}{8}\min {\left\{|x|^2, |x|^4\right\}},\forall x\in \mathbb{R}^2.
\end{align*}

Now, along the trajectory of system~\eqref{nonlinear ODE-closed loop}, we deduce that
\begin{align}\label{eqq}
\dot{V}= & x_1 \dot{x}_1+\left(1+h_1(t)\right) x_2^3 \dot{x}_2+\frac{1}{4} h_1^{\prime}(t) x_2^4 \notag\\
= & x_1\left(-g_1(t) x_1+g(t) x_1^m-h_1(t) x_2^3\right)  +\left(1+h_1(t)\right) x_2^3\left(h_2(t) x_1-g_2(t) x_2+{\tilde{g}(t)x^n_2}-b(t) x_1^2 x_2+d(t)\right)  +\frac{1}{4} h_1^{\prime}(t) x_2^4 \notag\\
= & -g_1(t) x_1^2+g(t) x_1^{m+1}+{\tilde{g}(t)\left(1+h_1(t)\right) x_2^{n+3}}-h_1(t) x_1 x_2^3+\left(1+h_1(t)\right) h_2(t) x_1 x_2^3  -g_2(t)\left(1+h_1(t)\right) x_2^4 \notag\\
&-b(t)\left(1+h_1(t)\right) x_1^2 x_2^4+\left(1+h_1(t)\right) x_2^3 d(t)+\frac{1}{4} h^\prime_1(t) x_2^4
\end{align}

{Using the condition} \eqref{con-3} and the facts that $b> 0$  and $h_1\ge 0$, we have
\begin{align}\label{eq-1}
-h_1(t) x_1 x_2^3+\left(1+h_1(t)\right) h_2(t) x_1 x_2^3
=  -h_1(t) x_1 x_2^3+h_1(t) x_1 x_2^3=0,
\end{align}
\begin{align}\label{eq--1}
g(t)x_1^{m+1}+\tilde{g}(t)\left(1+h_1(t)\right) x_2^{n+3}\leq& \overline{g}(t)(1+M)\left(\left(x^2_1+x^2_2\right)^{\frac{m+1}{2}}+\left(x^2_1+x^2_2\right)^{\frac{n+3}{2}}\right)\notag\\
= &\overline{g}(t)(1+M){\left(|x|^{m+1}+|x|^{n+3}\right)},
\end{align}
and
\begin{align}\label{eq-2}
-b(t)\left(1+h_1(t)\right) x_1^2 x_2^4 \leq 0,
\end{align}
where ${\overline{g}(t):=\max\left\{g(t),\tilde{g}(t)\right\}}$. Meanwhile, by using the Young's inequality, we get
\begin{align}\label{eq-3}
x_2^3 d(t)
\leq  \left(\varepsilon x_2^4+\varepsilon^{-3}(d(t))^4\right),
\end{align}
where $\varepsilon$ is a positive constant  to be determined later.

In addition, by using \eqref{con-2} and \eqref{condition-h1}, we have
\begin{align}\label{eq-4}
-g_2(t)\left(1+h_1(t)\right) x_2^4 +\frac{1}{4} h_1^{\prime}(t) x_2^4
\leq & -\left(1+h_1(t)\right) x_2^4-g_1(t)\left(1+h_1(t)\right) x_2^4  +\frac{1}{4} h_1(t) x_2^4 \notag\\
\leq & -\frac{3}{4} h_1(t) x_2^4- x_2^4-g_1(t) x_2^4 \notag\\
\leq & -g_1(t) x_2^4- x_2^4.
\end{align}
Putting \eqref{eq-1}, \eqref{eq--1}, \eqref{eq-2}, \eqref{eq-3}, and \eqref{eq-4} into \eqref{eqq}, we have
\begin{align*}
\dot{V} \leq & -g_1(t)\left(x_1^2+x_2^4\right)+{\overline{g}(t)(1+M) {\left(|x|^{m+1}+|x|^{n+3}\right)}}+(1+M) \varepsilon^{-3}(d(t))^4+(-1+(1+M)\varepsilon)x^4_2.
\end{align*}
By Lemma \ref{le-5} and choosing $\varepsilon\in\left(0, \frac{1}{M+1}\right)$, we have \begin{align*}
\dot{V} 	
\leq& -\frac{1}{2}g_1(t)\min {\left\{|x|^2, |x|^4\right\}}+{\overline{g}(t)(1+M) \max{\left\{|x|^{m+1},|x|^{n+3}\right\}}}+(1+M) \varepsilon^{-3} |d(t) |^4.
\end{align*}

Let $\theta_1:=\frac{1}{2} \min \left\{s^2, s^4\right\}$, $\theta_2:={(1+M)\max\left\{s^{m+1},s^{n+3}\right\}}$, and $\phi:=(1+M)\varepsilon^{-3}s^4$ for $s\in \mathbb{R}_{\geq 0}$. Note that $m>3$ and $n>1$. It is clear that $\theta_2(s)\leq \theta_1(s)$ for sufficiently small $s>0$. According to Definition \ref{DLiISS}, we deduce that $V(t,x):=\frac{1}{2}x_1^2+\frac{1}{4}(1+h_1(t))x_2^4$ is  an LiISS Lyapunov function for system \eqref{nonlinear ODE-closed loop}. Then, by Theorem \ref{TLiISS}, we conclude that system \eqref{nonlinear ODE-closed loop} is LiISS.
\end{proof of Prop.ODE}

 \textbf{Numerical results}\quad In simulations, for   system \eqref{nonlinear ODE-closed loop}, i.e.,   system \eqref{nonlinear ODE} under the control law \eqref{control law} in the presence of disturbances,
we set
\begin{align*}
	m=4, \quad n=4, \quad b(t)= 0.08+0.03\sin(3\pi t) , \quad 	h_1(t)=1+0.1\sin(2\pi t),
\end{align*}
which satisfy $m,n>1$, $b\in C(\mathbb{R}_{\geq 0};\mathbb{R}_{>0}) $, and \eqref{condition-h1}. In addition, we set
\begin{align*}
	g_1(t)=\frac{5}{1+t}, \quad
	g_2(t)=1+g_1(t), \quad g(t)=\tilde{g}(t)=\frac{1}{1+t}, \quad  h_2(t)=\frac{h_1(t)}{1+h_1(t)},\quad 	 d(t)=A\left(0.6e^{- \sqrt[4]{t}  }+1.2\cos(\pi t)\right),
\end{align*}
which satisfy \eqref{con-1}, \eqref{con-2}, and \eqref{con-3}, where $A=\left\{0,3,4\right\}$ are used to describe the amplitude of disturbances.

	Figure \ref{fig1} shows that    system \eqref{nonlinear ODE} in open loop, i.e., $u_1\equiv u_2\equiv 0$, is not asymptotically stable at the origin even for   relatively small  initial data; whereas Fig. \ref{fig2}  (see also black solid and dashed {curves} in Fig.~\ref{fig4}) demonstrates that,  for the same initial data as in  Fig.~\ref{fig1},      system \eqref{nonlinear ODE} under the control law \eqref{control law}, i.e., system  \eqref{nonlinear ODE-closed loop} in the absence of disturbances, is  asymptotically stable at the origin.
	In the presence of different disturbances,  Fig. \ref{fig2},  Fig. \ref{fig3},  Fig. \ref{fig3'},  and Fig. \ref{fig4}   show that the solution {to} {system}  \eqref{nonlinear ODE-closed loop}  remains bounded under  variation of  disturbances and relatively small initial data.
	In particular, it is shown in  Fig. \ref{fig3} (a) and (b) (or Fig.~\ref{fig3'} (a) and (b)) that for the same initial datum,   the amplitude  of the states  of   system \eqref{nonlinear ODE-closed loop} decreases as the amplitude of the disturbances diminishes, while it is shown in  Fig. \ref{fig3} (a) and Fig.~\ref{fig3'} (a)  (or Fig. \ref{fig3} (b) and Fig.~\ref{fig3'} (b)) that for the same disturbances,  the amplitude  of the states  of   system \eqref{nonlinear ODE-closed loop} decreases as the  initial datum becomes small.
	Corresponding to the Fig. \ref{fig2},  Fig. \ref{fig3}, and Fig. \ref{fig3'}, the evolution of the states' norms for system \eqref{nonlinear ODE-closed loop} are depicted in Fig.~\ref{fig4}. Especially,  the black solid and dashed {curves}  illustrate  that the states' norms of system \eqref{nonlinear ODE-closed loop}   converge   rapidly to   the origin in the absence of disturbances for different initial data,  and   at any fixed time   the bound of the  states' norms diminishes as the initial value decreases; whereas the red and blue solid {curves} (or the red and blue dashed {curves}) illustrate that, for the same initial datum, the bound of the states' norms of   system \eqref{nonlinear ODE-closed loop} decreases as the amplitude of the disturbances diminishes. Notably, the LiISS property implies that all trajectories of system \eqref{nonlinear ODE-closed loop} with differently relatively small initial data
	are driven by a uniform decay rate; consequently, as time tends to infinity, the ultimate bound of  the states is  determined primarily by the disturbances rather than by the initial data.  This property is reflected by both the solid and dashed {curves}, indicating that despite different initial data,   the norms of the states admit  {almost the same bound} when the time is sufficiently large, provided the disturbances are identical. Note that for relatively large initial data, system \eqref{nonlinear ODE-closed loop} may blow up in finite time as shown in Fig. \ref{fig5}. Overall, Fig. \ref{fig1}, Fig.\ref{fig2}, Fig.\ref{fig3},  Fig. \ref{fig3'},   Fig.\ref{fig4}, and Fig.\ref{fig5} collectively well illustrate the LiISS property of   system \eqref{nonlinear ODE-closed loop}.

	\begin{figure*}[htbp!]
		\begin{center}
			\subfigure[{Evolution~of~$x$~when~${(x_1(0),x_2(0))^{\top}=(0.1,0.25)^{\top}}$.}]{ \includegraphics[width=0.35\textwidth,height=7pc]{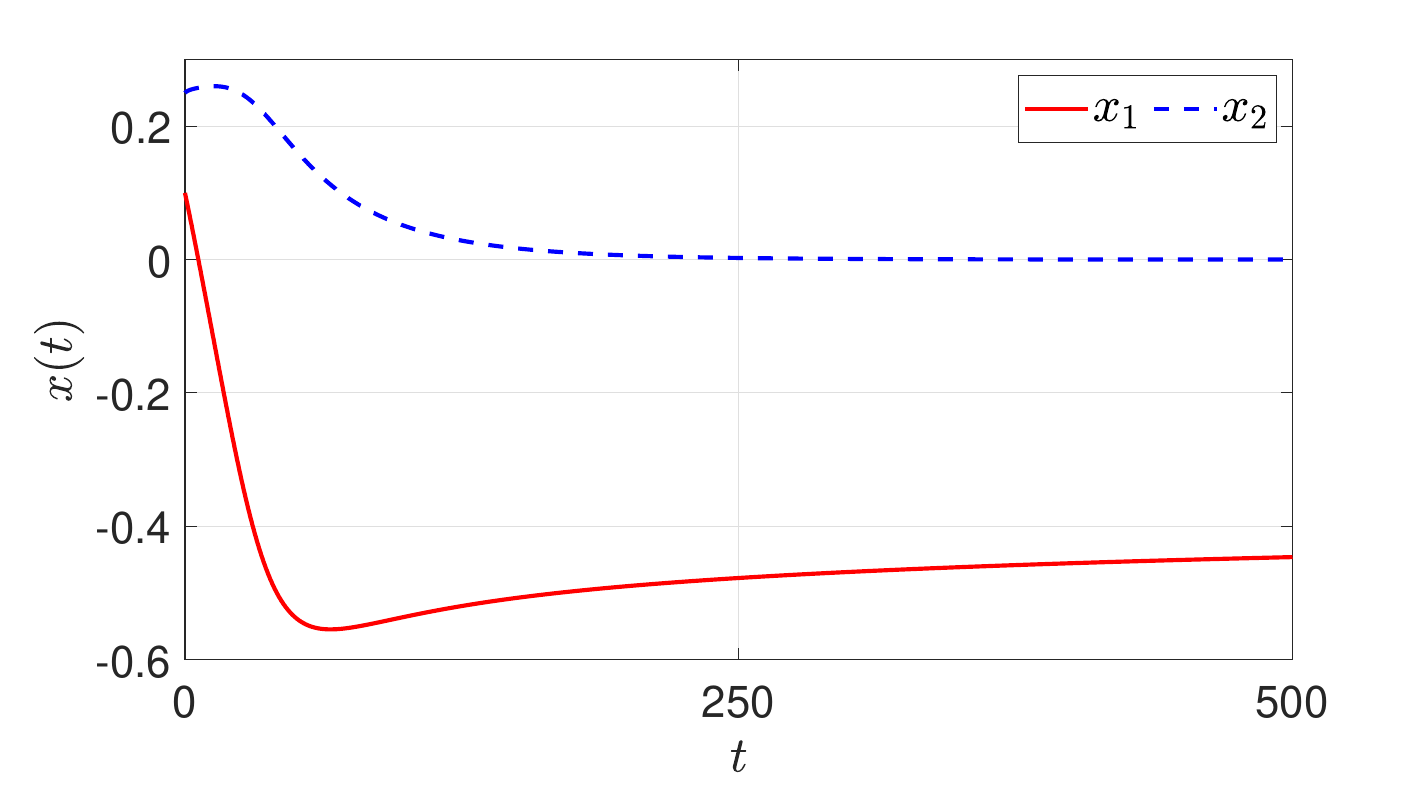}}\hspace{50pt}
			\subfigure[{Evolution~of~$x$~when~${(x_1(0),x_2(0))^{\top}=(0.2,0.5)^{\top}}$.}]{ \includegraphics[width=0.35\textwidth,height=7pc]{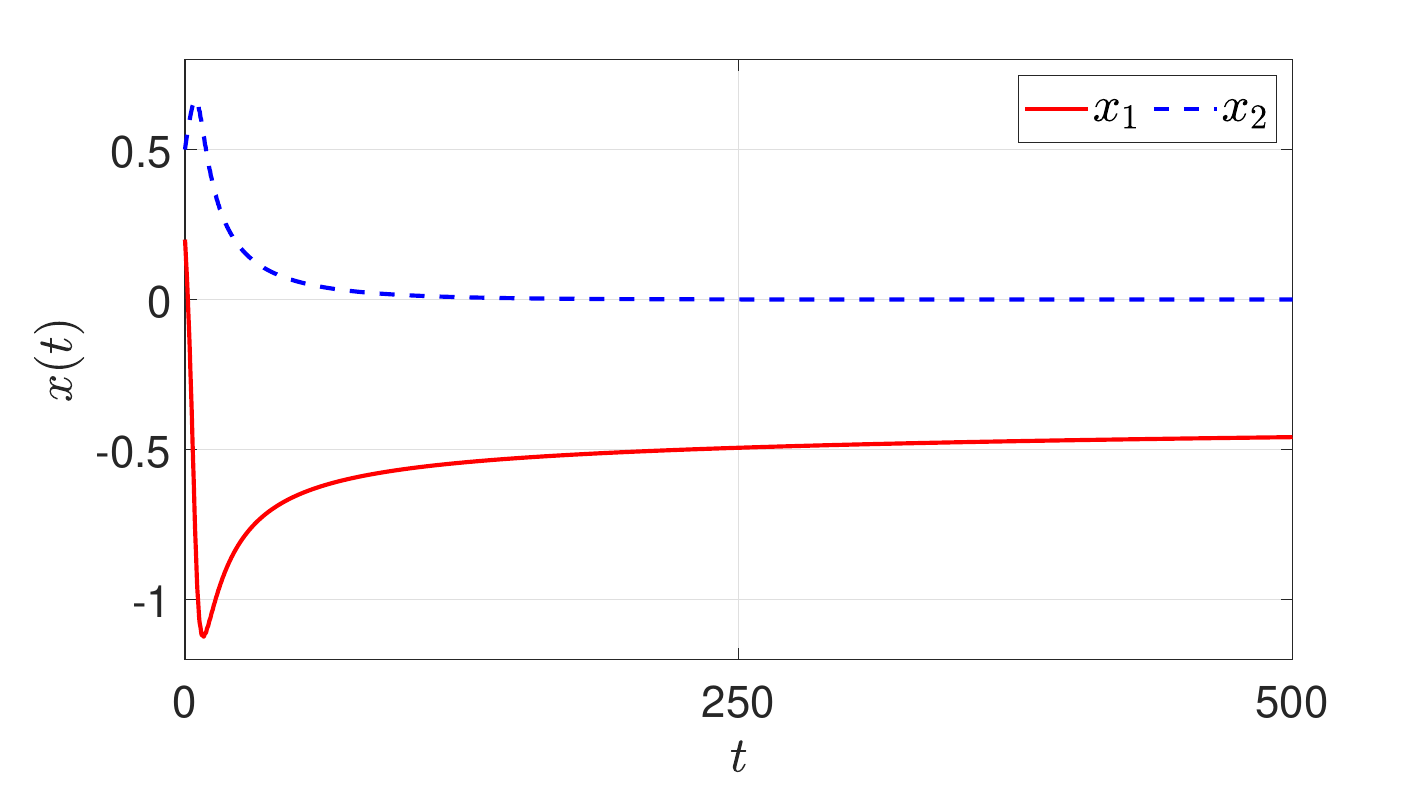}}
			\subfigure[{Evolution~of~${|x|}$~when~${(x_1(0),x_2(0))^{\top}=(0.1,0.25)^{\top}}$.}]{ \includegraphics[width=0.35\textwidth,height=7pc]{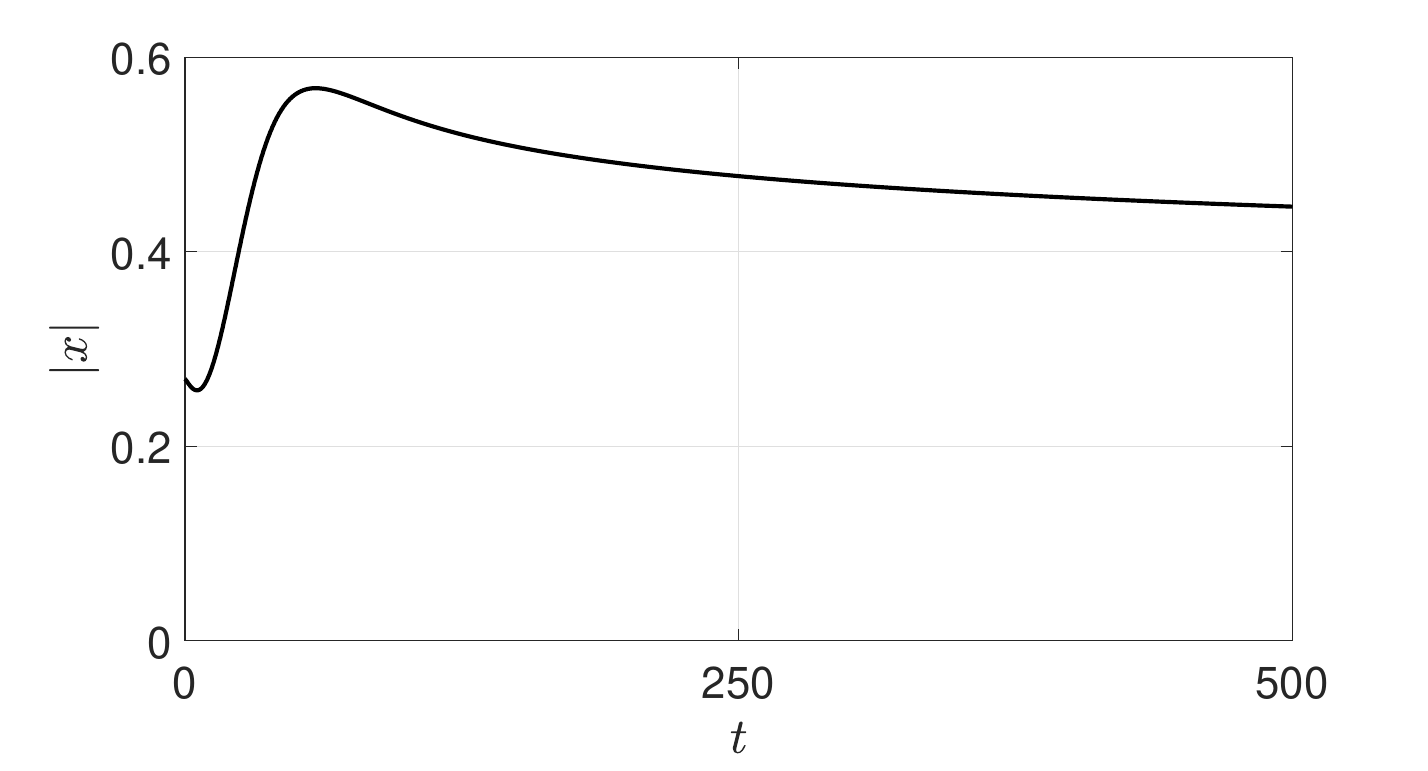}}\hspace{50pt}
			\subfigure[{Evolution~of~${|x|}$~when~${(x_1(0),x_2(0))^{\top}=(0.2,0.5)^{\top}}$.}]{ \includegraphics[width=0.35\textwidth,height=7pc]{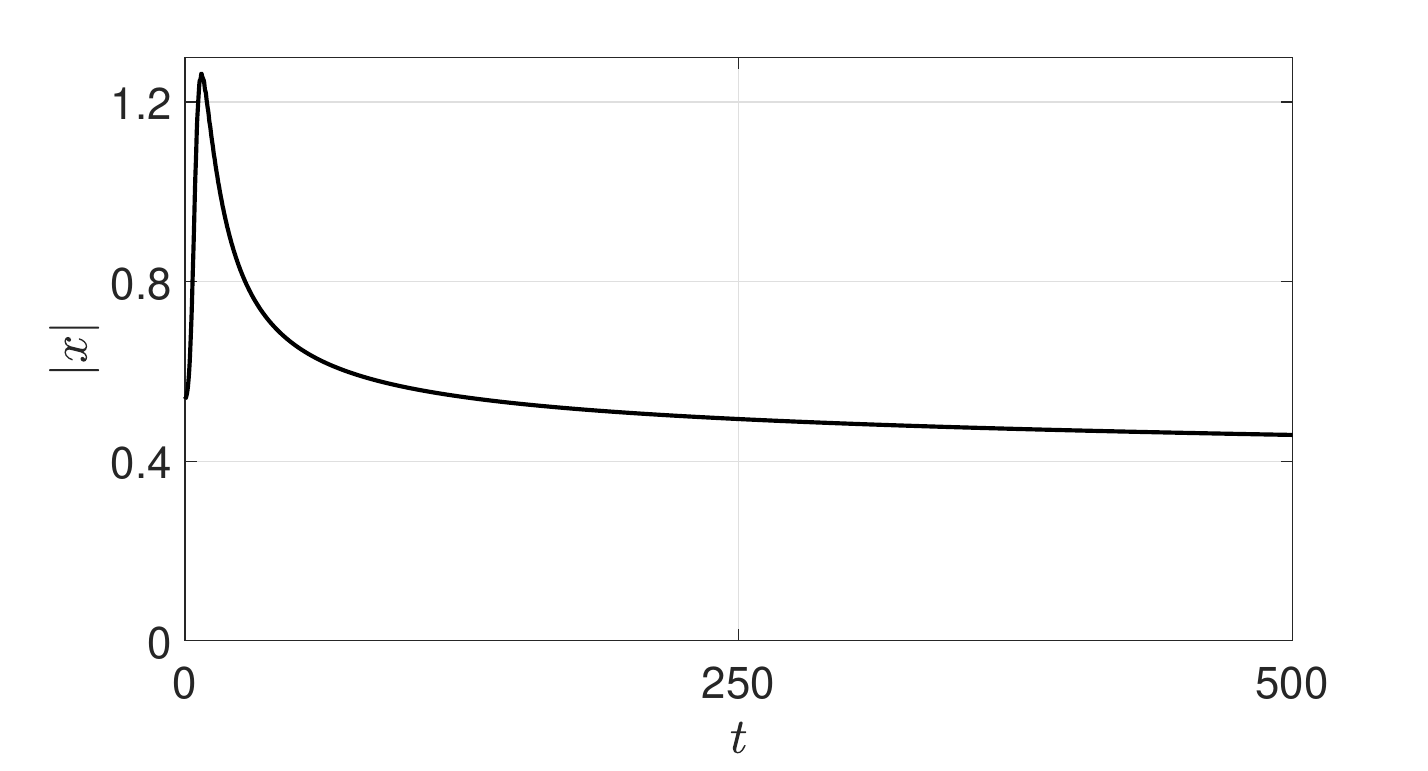}}
			\caption{Evolution of     $x$ and ${|x|}$ for system \eqref{nonlinear ODE}  in open loop with different small initial data. \label{fig1}}
		\end{center}
	\end{figure*}

\begin{figure*}[htpb!]
	\begin{center}
		\subfigure[{Evolution~of~$x$~when~${(x_1(0),x_2(0))^{\top}=(0.1,0.25)^{\top}}$.}]{
			\includegraphics[width=0.35\textwidth,height=7pc]{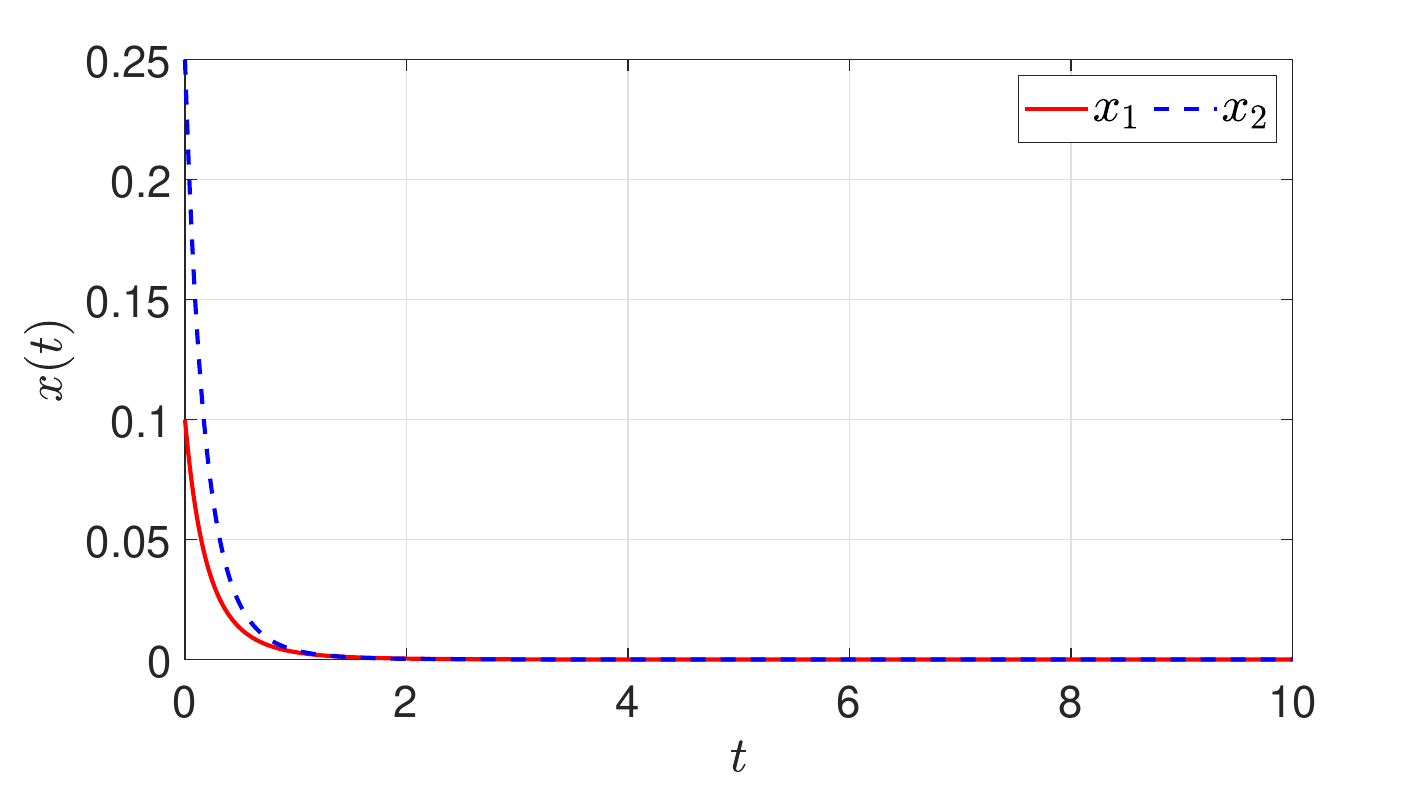}}\hspace{50pt}
		\subfigure[{Evolution~of~$x$~when~${(x_1(0),x_2(0))^{\top}=(0.2,0.5)^{\top}}$.}]{
			\includegraphics[width=0.35\textwidth,height=7pc]{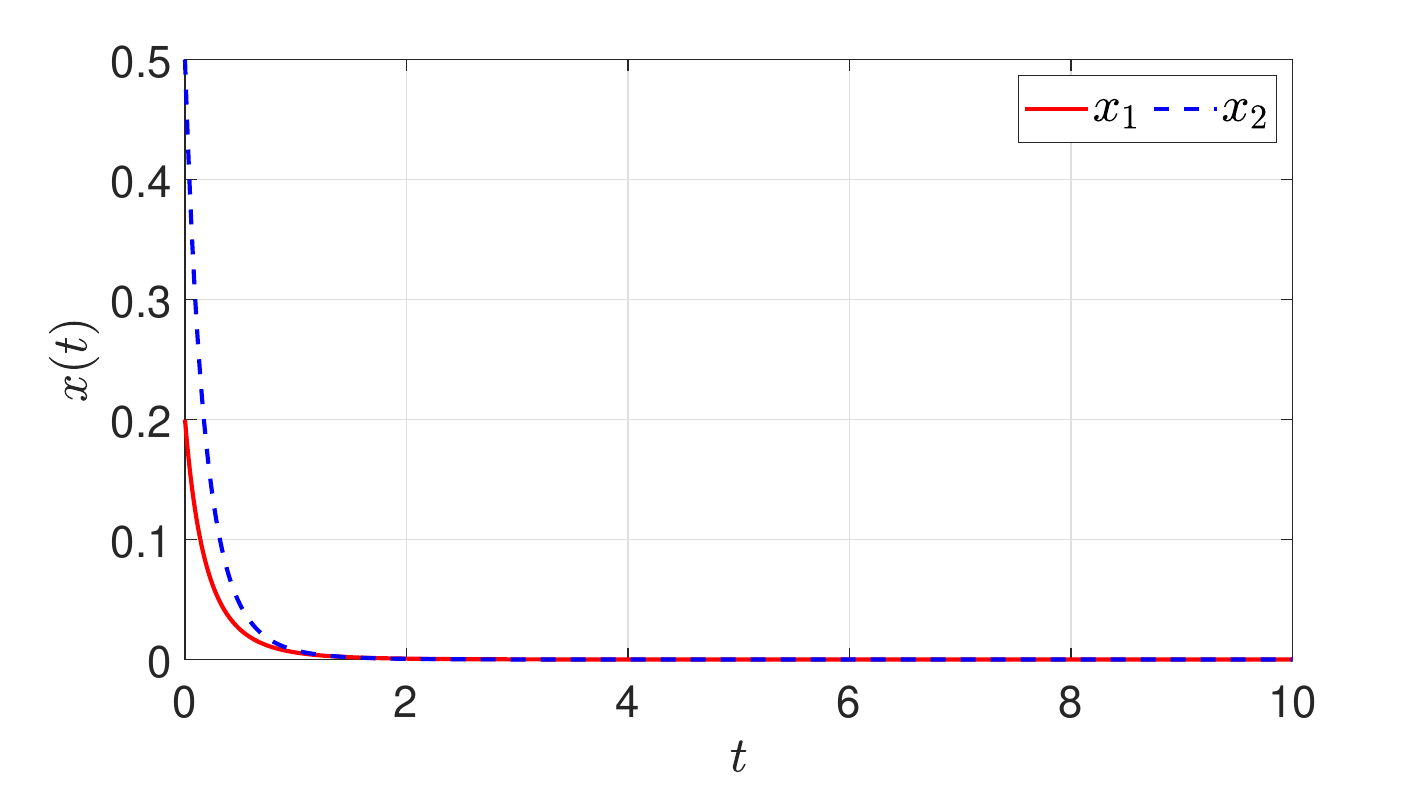}}
		\caption{{Evolution of $x$} for     system \eqref{nonlinear ODE} under the control law \eqref{control law}   with different small initial data.\label{fig2}}
	\end{center}
\end{figure*}

\begin{figure*}[htpb!]
	\begin{center}
		\subfigure[{Evolution of $x$ when $A=3$.}]{
			\includegraphics[width=0.35\textwidth,height=7pc]{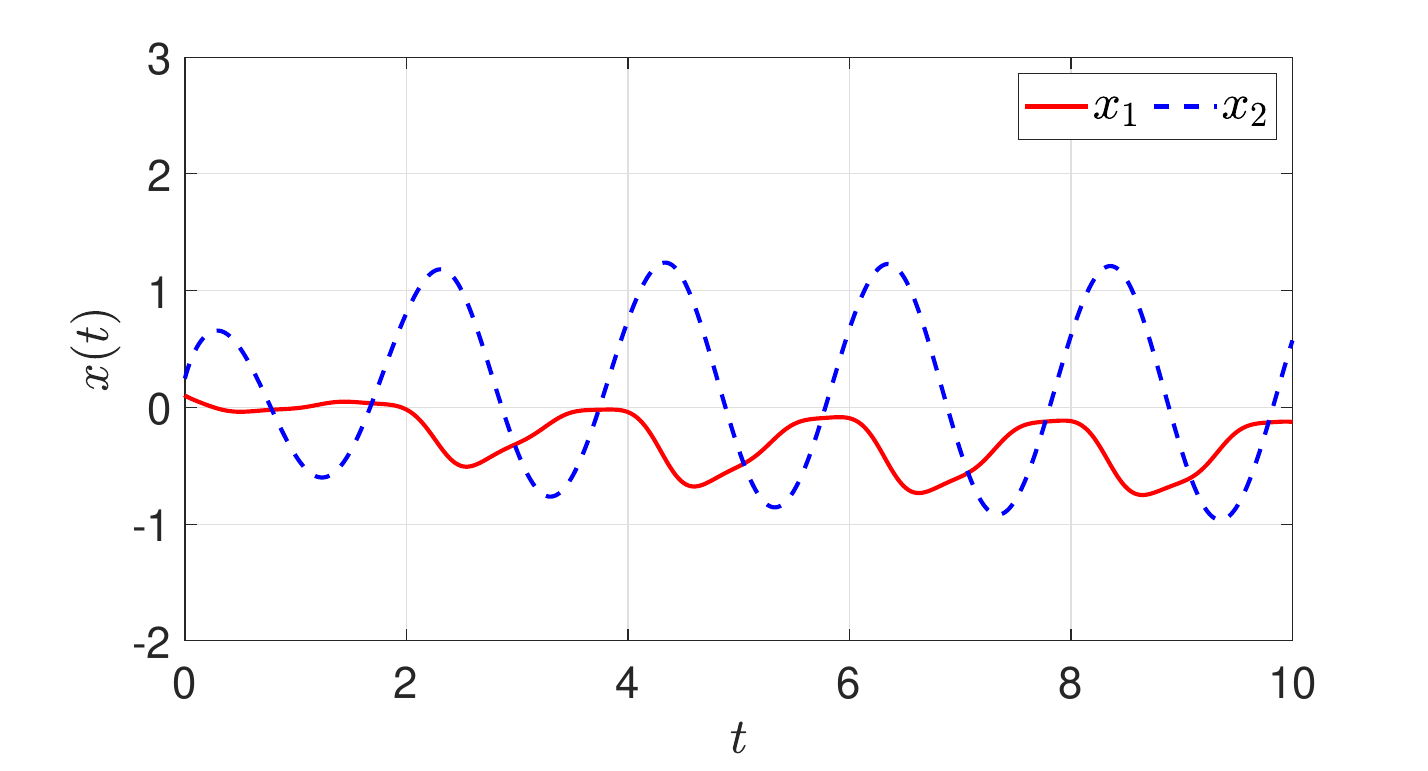}}\hspace{50pt}
		\subfigure[{Evolution of $x$ when $A=4$.}]{
			\includegraphics[width=0.35\textwidth,height=7pc]{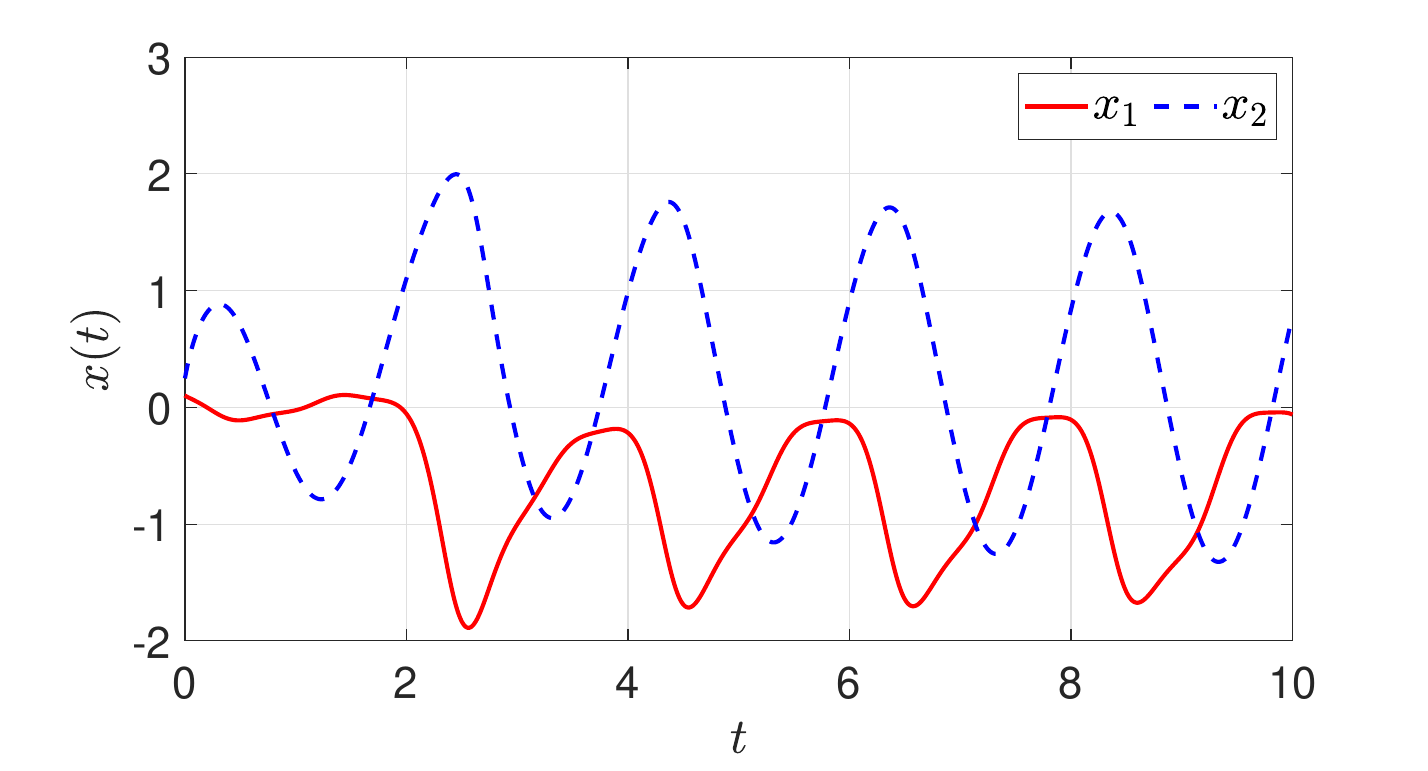}}
		\caption{Evolution of $x$  for system \eqref{nonlinear ODE-closed loop} with different disturbances when   ${(x_1(0),x_2(0))^{\top}=(0.1,0.25)^{\top}}$.\label{fig3}}
	\end{center}
\end{figure*}

\begin{figure*}[htpb!]
	\begin{center}
		\subfigure[{Evolution of $x$ when $A=3$.}]{
			\includegraphics[width=0.35\textwidth,height=7pc]{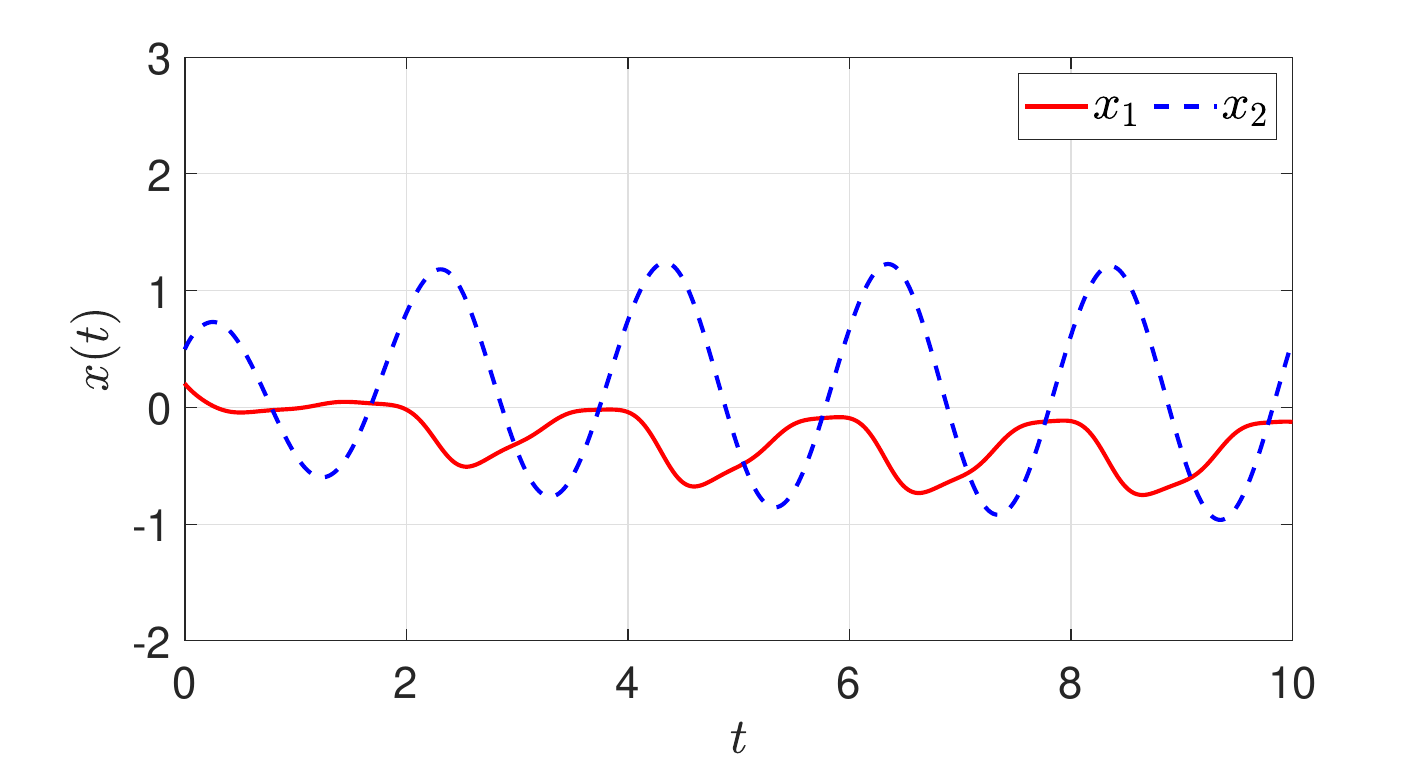}}\hspace{50pt}
		\subfigure[{Evolution of $x$ when $A=4$.}]{
			\includegraphics[width=0.35\textwidth,height=7pc]{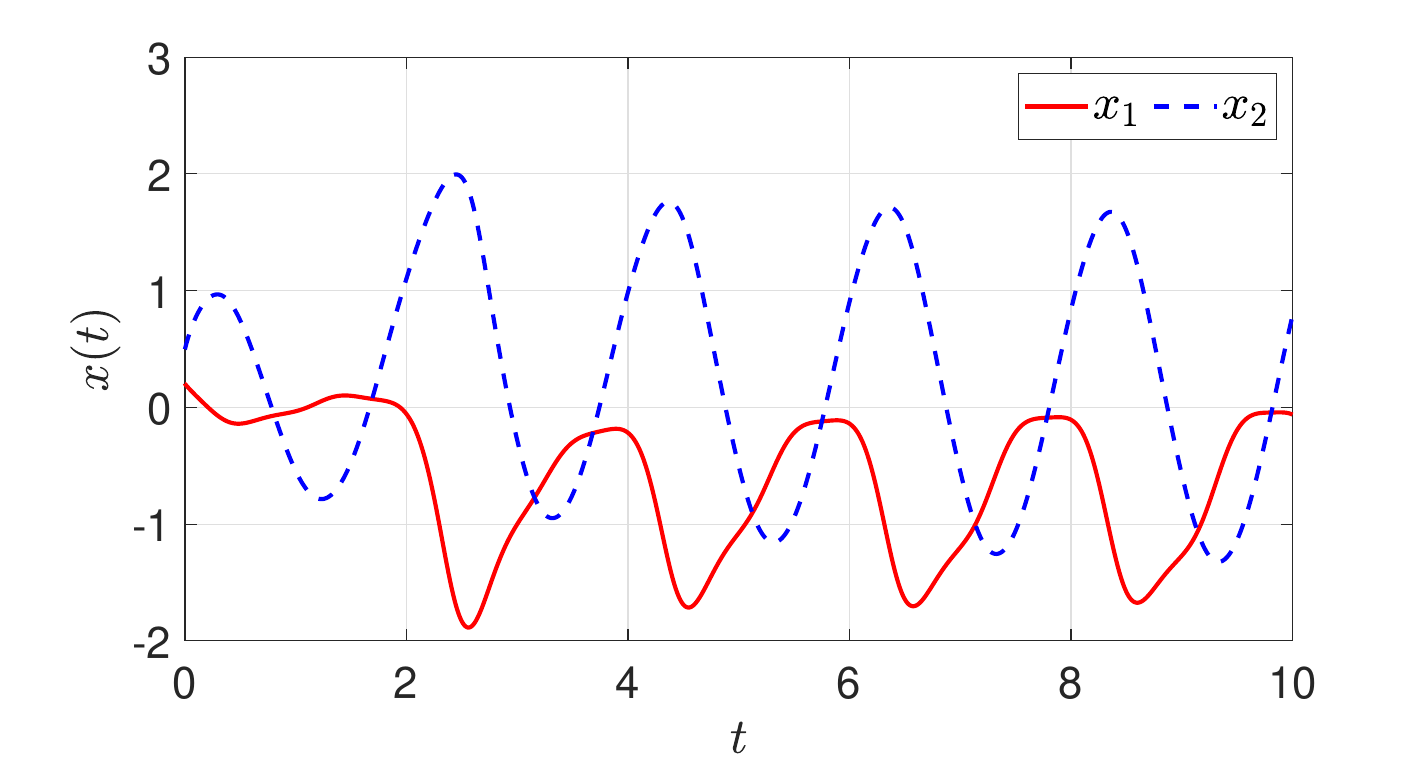}}
		\caption{Evolution of $x$  for system \eqref{nonlinear ODE-closed loop} with different disturbances when ${(x_1(0),x_2(0))^{\top}=(0.2,0.5)^{\top}}$.\label{fig3'}}
	\end{center}
\end{figure*}

\begin{figure*}[htpb!]
	\begin{center}
		\includegraphics[width=0.35\textwidth,height=7pc]{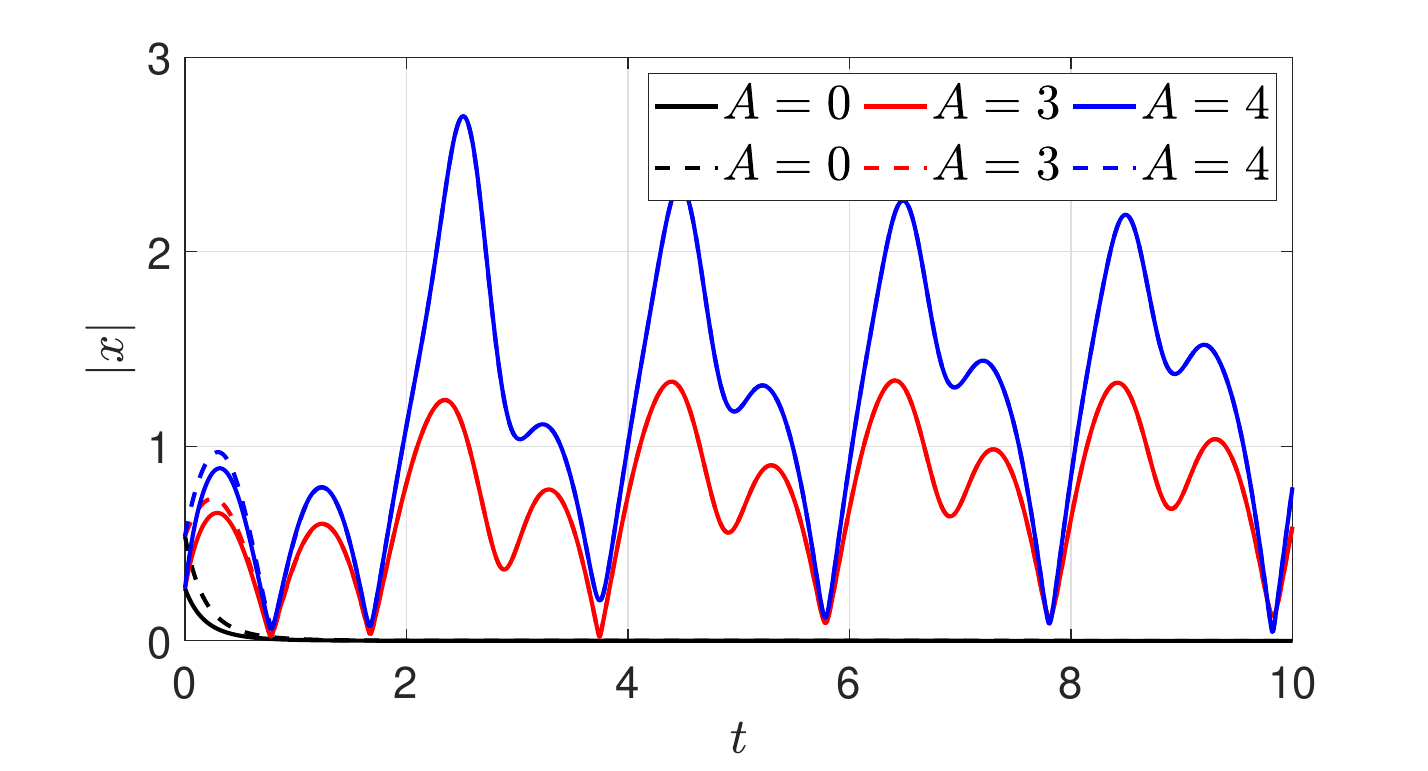}
		\caption{Evolution of   ${|x|}$ for system \eqref{nonlinear ODE-closed loop} with different disturbances and relatively small initial data:  solid {curve} for  ${(x_1(0),x_2(0))^{\top}=(0.1,0.25)^{\top}}$ and dashed {curve} for  ${(x_1(0),x_2(0))^{\top}=(0.2,0.5)^{\top}}$.\label{fig4}}
	\end{center}
\end{figure*}

\begin{figure*}[htpb!]
	\centering
	\subfigure[{Evolution of $x$.}]{ \includegraphics[width=0.35\textwidth,height=7pc]{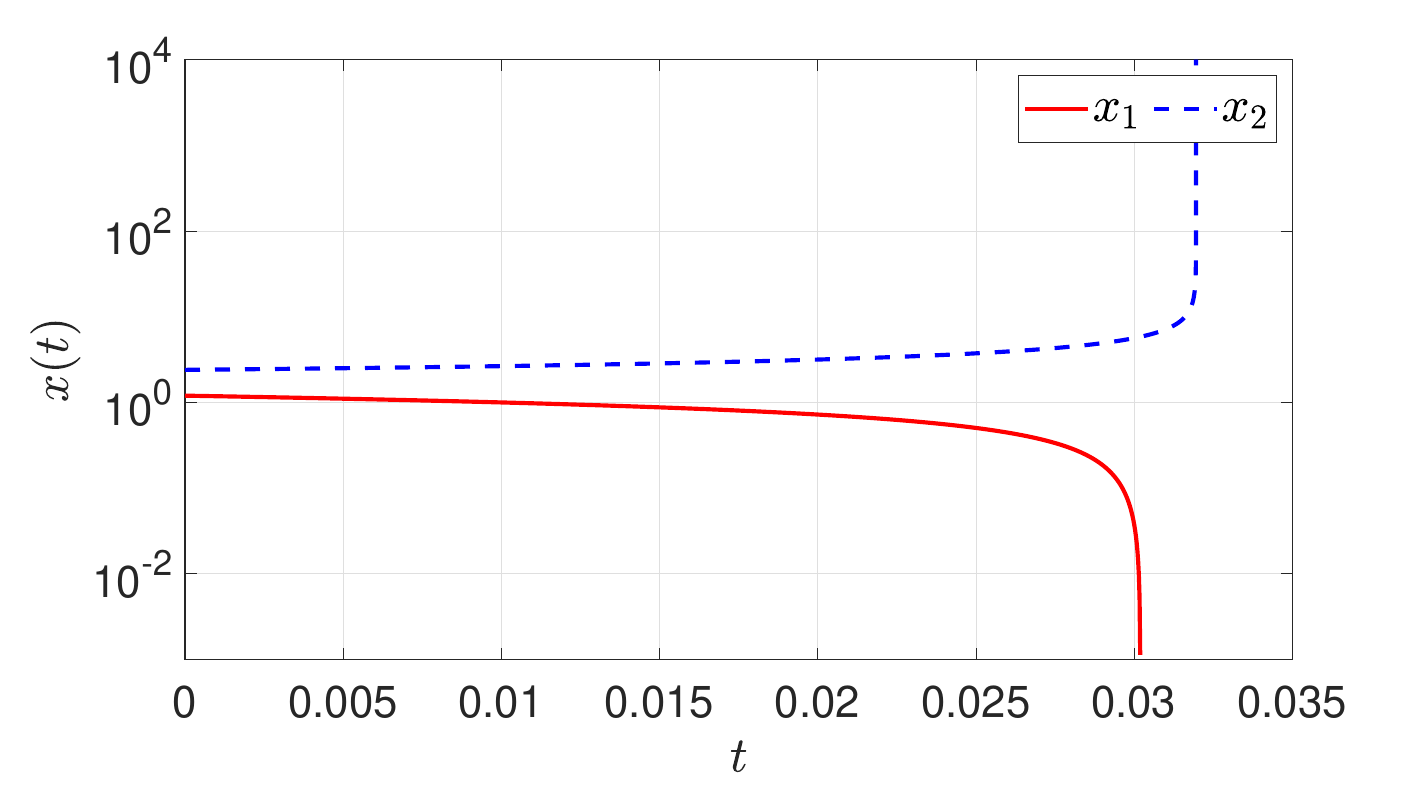}}\hspace{50pt}	\subfigure[{Evolution of ${|x|}$.}]{ \includegraphics[width=0.35\textwidth,height=7pc]{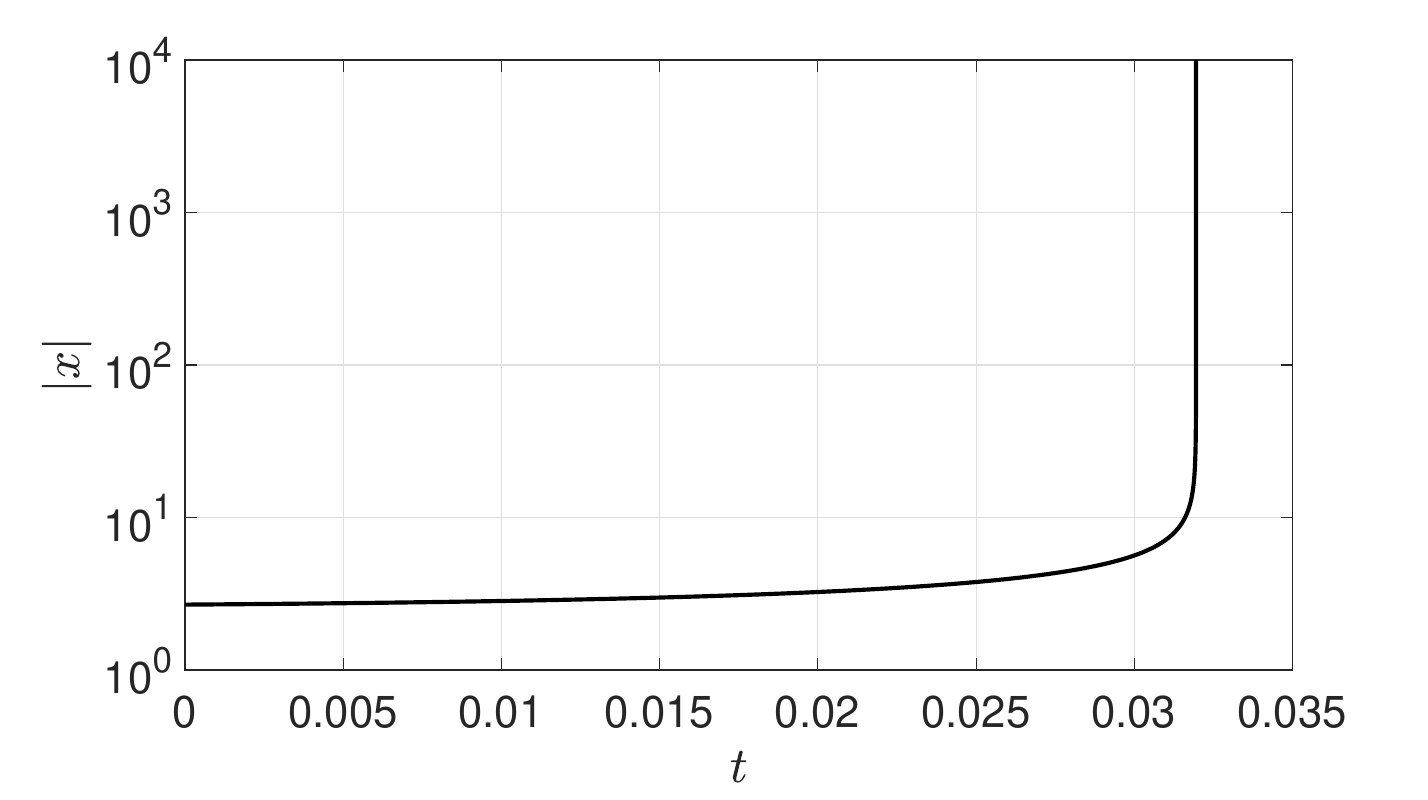}}
	\caption{ Evolution of $x$ and ${|x|}$ for
		system \eqref{nonlinear ODE-closed loop} in the absence of disturbances with relatively large initial data
		${(x_1(0),x_2(0))^{\top}=(1.2,2.4)^{\top}}$.\label{fig5}}
\end{figure*}

\subsection{LiISS of time-varying parabolic PDEs with  superlinear   terms}\label{sec:5.2}


In this section, as an application of the LiISS Lyapunov theorem, we  investigate  the LiISS of   $N$-dimensional parabolic PDEs with time-varying coefficients and superlinear terms in the following form:
\begin{subequations}\label{4.4}
	\begin{align}
		\frac{\partial x}{\partial t}(\xi,t)-{\rm div}\ (a(\xi,t)\nabla x(\xi,t))+c(\xi,t)x(\xi,t)= & u(\xi,t)+h(\xi,t,x(\xi,t)),(\xi,t)\in \Omega\times\mathbb{R}_{\geq 0},\label{3.1a}\\
		a(\xi,t)\frac{\partial x }{\partial \nu} (\xi,t)= & 0,(\xi,t)\in\partial \Omega\times \mathbb{R}_{\geq 0},\\
		x(\xi,0)= & x_0(\xi),\xi\in\Omega,
	\end{align}
\end{subequations}
where $\Omega $ is an open bounded domain in $\mathbb{R}^N  (N\in\mathbb{N}) $ with $C^1$-continuous boundary $\partial \Omega$,    $a\in C (\Omega\times\mathbb{R}_{\geq 0};\mathbb{R}_{>0})$ and $c\in C(\Omega\times\mathbb{R}_{\geq 0};\mathbb{R}_{>0})$ are  functions depending on space-time variables,    $u\in C (\Omega\times\mathbb{R}_{\geq 0};\mathbb{R})$ represents in-domain disturbances, $h\in C(\Omega\times\mathbb{R}_{\geq 0}\times\mathbb{R};\mathbb{R})$ is a nonlinear function having superlinear growth w.r.t. the third variable,    $\operatorname{div}$ denotes the standard divergence operator, and $\nu$ is the outer unit normal vector.

Assume that
\begin{align*}
	a(\xi, t)\geq &\underline{a}   ,\forall (\xi, t)\in{\Omega}\times\mathbb{R}_{\geq 0}, \\
	c(\xi, t)\geq &   \underline{c}  , \forall (\xi, t)\in{\Omega}\times\mathbb{R}_{\geq 0}, \\
	\left|h(\xi, t, x)\right|\leq & M( \left|x\right|^{m_1}+  \left|x\right|^{m_2}), \forall (\xi, t,x)\in{\Omega}\times\mathbb{R}_{\geq 0}\times \mathbb{R},
\end{align*}
where  $\underline{a} ,\underline{c},M,m_1$, {and $m_2$} are positive constants, among which $m_1$ and $m_2$  satisfy either of the following  conditions:

\textbf{(H1)} For $1\leq N\leq 2$,  assume that $1<m_i<+\infty $,  $i=1,2$;

\textbf{(H2)} For $N>2$, assume that $ {1<m_i\leq \frac{N}{N-2}}, i=1, 2$.

Let $ X =L^{2}(\Omega),  U=C(\Omega)$. For all $t\in\mathbb{R}_{\geq 0}$, introducing the operator $A(t):D(A)\subset X \rightarrow X $ by $A(t)x:=  {\rm div}(a\nabla x)$, where
\begin{align*}
	D(A(t)):=\left\{x\in  X \Big|~{\rm div}(a\nabla x)\  \text{exists\ in}\ \Omega\times\mathbb{R}_{\geq 0}, \  a(\xi, t)\frac{\partial x }{\partial {\nu}}(\xi,t)  =0\ \text{on}\ \partial \Omega\times \mathbb{R}_{\geq 0}\right\}.
\end{align*}
{Let $F(t,x,u)=-c(\xi,t)x+h(\xi,t,x)+u$.
	It is easy to write   system~\eqref{4.4} in an abstract form as \eqref{systems}. Moreover, for $X_{\frac{1}{2}} =W^{1,2}(\Omega)$,  noting that    $W^{1,2}(\Omega)\subset L^{q}(\Omega)$  for any $q\in [1,+\infty)$ and $N=1$ or $N=2$, and that $W^{1,2}(\Omega)\subset L^{\frac{2N}{N-2}}(\Omega)$ for $N>2$,  each of the conditions \textbf{(H1)} and \textbf{(H2)} ensures that $|x|^{m_i}\in L^{2}(\Omega), i=1,2,$ for $x\in X_{\frac{1}{2}}$. Therefore,  for any fixed $t$, $F$ maps the state $x$ from $X_{\frac{1}{2}}$ to $X$.}
Furthermore, By \cite[Theorem 6.2 p.~457 and Theorem 7.4, p.~491]{Lady1968linear}, system~\eqref{4.4} admits at least one classical solution provided that $x_0\in D(A(0))$; see {\cite[\S 7, Chap. V]{Lady1968linear}}.

For the parabolic system~\eqref{4.4}, we have the following LiISS result.
\begin{proposition} \label{Prop.8}
	System~\eqref{4.4} admits an LiISS-LF and hence, is LiISS.
\end{proposition}



To prove Proposition \ref{Prop.8}, we need the following Sobolev embedding and interpolation inequalities, which are used to deal with superlinear terms in the LiISS analysis.
\begin{lemma}[{\cite[pp. 285, 313, 314, 119]{Brezis2011function}}] \label{interpolation}
	The following inequalities hold true:
	\begin{enumerate}
		\item[(i)] if $ 1\leq q<N$ and $v\in W^{1,q}(\Omega)$, then $\|v\|_{q^*}\leq C\|v\|_{1,q}$, where $q^*:=\frac{qN}{N-q}$ {and $C$ is a positive constant};
		\item[(ii)] if  $ 1\leq q\leq p< +\infty$, $r\geq N$, and $v\in W^{1,r}(\Omega)$, then $\|v\|_p\leq C\|v\|_q^{1-{\lambda}}\|v\|_{1,r}^{{\lambda}}$,
		where ${\lambda}=\frac{\frac{1}{q}-\frac{1}{p}} {\frac{1}{q}+\frac{1}{N}-\frac{1}{r}}\in [0,1)$ and   $C$ is a positive constant;
		\item[(iii)] if $ 1\leq q\leq p\leq r\leq +\infty$ and $v\in L^{r}(\Omega)$, then $\|v\|_p\leq \|v\|_q^{1-{\lambda}}\|v\|_{r}^{{\lambda}}$,
		where ${\lambda}\in [0,1]$ {satisfies} $\frac{1}{p}=\frac{{\lambda}}{r}+\frac{1-{\lambda}}{q}$.
	\end{enumerate}
\end{lemma}

Using Lemma \ref{interpolation}, we can prove Proposition \ref{Prop.8}.

\begin{proof of PDE}
	Define
	\begin{align*}
		V(x):=\frac{1}{2}\|x(\cdot, t)\|^{2}=\frac{1}{2}\int_{\Omega}x^{2}(\xi, t){\rm d}\xi.
	\end{align*}
	
	By integrating by parts,  and using the Young's inequality and the structural conditions on $a,c,h$, we deduce that the derivative of $V$ along the trajectory satisfies
	\begin{align}
		\!\!{\dot{V}(x)}= &\int_{\Omega}x \left({\rm div}(a \nabla x )-c x +u +h \right){\rm d}\xi\notag\\
		%
		\!\!=&-\int_{\Omega}a|\nabla x|^{2}{\rm d}\xi  -\int_{\Omega} c x^2 {\rm d}\xi  +\int_{\Omega}  ux{\rm d}\xi +\int_{\Omega} h x{\rm d}\xi\notag\\
		\!\!\leq &-\underline{a} \|\nabla x\|^2  -  \underline{c}   \|x\|^2 +\frac{\varepsilon}{2} \|x\|^2 +\frac{1}{2\varepsilon}   \|u(\cdot,t)\|^2 + M\left(\|x\|_{m_1+1}^{m_1+1}+ \|x\|_{m_2+1}^{m_2+1}\right)\notag\\
		\!\!\leq &-\underline{a} \|\nabla x\|^2  -  \underline{c}   \|x\|^2 +\frac{\varepsilon}{2} \|x\|^2 +\frac{1}{2\varepsilon}    {\|u(\cdot,t)\|^2_{U}} + M\left(\|x\|_{m_1+1}^{m_1+1}+ \|x\|_{m_2+1}^{m_2+1}\right),\label{16}
	\end{align}
	where $\varepsilon $ is a positive constant to be determined later.
	
	
	Now we prove in two cases.
	
	\textbf{Case 1}:  $1\leq N\leq 2, 1<m_{i}<+\infty, i=1, 2$.
	
	Let
	\begin{align*}
		p=m_{i}+1, q=r=2, {\lambda_{i}}=N\left(\frac{1}{2}-\frac{1}{m_{i}+1}\right)\in(0, 1).
	\end{align*}
	
	By Lemma~\ref{interpolation} (ii) and  the Young's inequality,  we have
	\begin{align}
		\|x\|^{m_{i}+1}_{{m_{i}+1}}\leq &  C\|x\|^{(1-{\lambda_{i}})(m_{i}+1)}\|x\|_{1, 2}^{{\lambda_{i}}(m_{i}+1)}\notag\\
		= & C\|x\|^{(1-{\lambda_{i}})(m_{i}+1)} {(\|  x\|+\| \nabla x\|)^{{\lambda_{i}}(m_{i}+1)}}\notag\\
		\leq &  C' {\left( \|x\|^{ m_{i}+1 }+ \|x\|^{(1-{\lambda_{i}})(m_{i}+1)}\| \nabla x\| ^{{\lambda_{i}}(m_{i}+1)}\right)}\notag\\
		\leq &  C'  \left(\|x\|^{ m_{i}+1 } + C(\varepsilon) \|x\|^{{n_i}}+{\varepsilon}\| \nabla x\|^{2} \right)\notag\\
		\leq &   C'' \left(\|x\|^{ m_{i}+1 } +\|x\|^{{n_i}}\right)+ \varepsilon C' \| \nabla x\|^{2}, i=1, 2,\label{15}
	\end{align}
	{where   $C,C'$, {and $C''$} are positive   constants, $C(\varepsilon):={\varepsilon^{-\frac{\lambda_i(m_i+1)}{2-\lambda_i(m_i+1)}}}$, and $n_i:=\frac{2(1-{\lambda_{i}})(m_{i}+1)}{2-{\lambda_{i}}(m_{i}+1)}>2, i=1,2$.}
	
	Putting \eqref{15} into \eqref{16}, we get
	\begin{align*}
		{\dot{V}(x)}
		\leq &-{(\underline{a} -2M\varepsilon C')}\|\nabla x\|^2  - \left (\underline{c}  - \frac{\varepsilon}{2} \right)\|x\|^2 + C ''M \left(\|x\|^{ m_{1}+1 } +\|x\|^{{n_1}}\right)
		+ C ''M \left(\|x\|^{ m_{2}+1} +\|x\|^{{n_2}}\right) +\frac{1}{2\varepsilon}   \|u(\cdot,t)\|^2_U \notag\\
		\leq &   {- \left (\underline{c}  - \frac{\varepsilon}{2} \right)\|x\|^2 +C''M \left(\|x\|^{m} +\|x\|^{n} \right) +\frac{1}{2\varepsilon}   \|u(\cdot,t)\|^2_U},
	\end{align*}
	where  {$m:= \min\{ m_{1}+1,n_1,m_{2}+1,n_2\}>2,
		n:= \max\{ m_{1}+1,n_1,m_{2}+1,n_2\}>2$,} and we {chose} $ 0<\varepsilon<\min\left\{2\underline{c},{\frac{\underline{a}}{2MC'}}\right\}$.
	%
	%
	%

	Therefore, the conditions of Theorem \ref{Thm2} are fulfilled. We deduce  that $V( {x})=\|x\|^{2}$ is {an LiISS-LF}. Hence, system~\eqref{4.4} is LiISS.

	\textbf{Case 2}:  $N > 2$,  $2 < m_{i} + 1 < 2^{*}$,  $i = 1$, $2$.
	
	Let
	\begin{align*}
		p=m_{i}+1, q=2, r=2^{*}, {\lambda_{i}}=\frac{\frac{1}{2}-\frac{1}{m_{i}+1}}{\frac{1}{2}-\frac{1}{2^{*}}}\in(0, 1).
	\end{align*}

	{Note that $2<m_i+1\leq \frac{2N-2}{N-2}< 2^*$.} Analogous to the proof of \eqref{15}, using   Lemma~\ref{interpolation}  (iii) and (i),  and the Young's inequality with $\varepsilon>0$, we obtain
	\begin{align*}
		\|x\|^{m_{i}+1}_{{m_{i}+1}}\leq &  C\|x\|^{(1-{\lambda_{i}})(m_{i}+1)}\|x\|_{2^{*}}^{{\lambda_{i}}(m_{i}+1)}\\
		\leq &  C'\|x\|^{(1-{\lambda_{i}})(m_{i}+1)}\|x\|_{1, 2}^{{\lambda_{i}}(m_{i}+1)}\\
		= & C'\|x\|^{(1-{\lambda_{i}})(m_{i}+1)}{(\|  x\|+\| \nabla x\|)^{{\lambda_{i}}(m_{i}+1)}}\notag\\
		\leq &   C'' \left(\|x\|^{ m_{i}+1 } +\|x\|^{{n_i}}\right)+ \varepsilon C'' \| \nabla x\|^{2}, i=1, 2,
	\end{align*}
	{where   $C,C'$, {and $C''$} are positive   constants    and $n_i:=\frac{2(1-{\lambda_{i}})(m_{i}+1)}{2-{\lambda_{i}}(m_{i}+1)}>2, i=1,2$.}
	
	Then, as in   Case 1, we can prove that   $V( {x})=\|x\|^{2}$ is an LiISS-LF of system~\eqref{4.4}. Therefore, system~\eqref{4.4} is LiISS. 
\end{proof of PDE}

\textbf{Numerical results}\quad In simulations, for system \eqref{4.4}, we set
\begin{align*}
	N=&1,\\
	a(\xi,t)=&\sqrt{e^{-t}+0.5\sin(\pi\xi)+0.5}, \\
	c(\xi,t)=&1+\pi\sqrt{\sin\xi+e^{-t}+1}, \\
	h(\xi,t,x)=&\frac{\sqrt{1+\sin\xi}}{1+t}x^2,\\
	u(\xi,t)=&A_1\sin(10t+x),
\end{align*}
where   $A_1=\left\{0,1,3\right\}$ are used to describe the amplitude of disturbances.


\begin{figure*}[htpb!]
	\begin{center}
		\subfigure[{Evolution of $x$.}]{
			\includegraphics[width=0.35\textwidth]{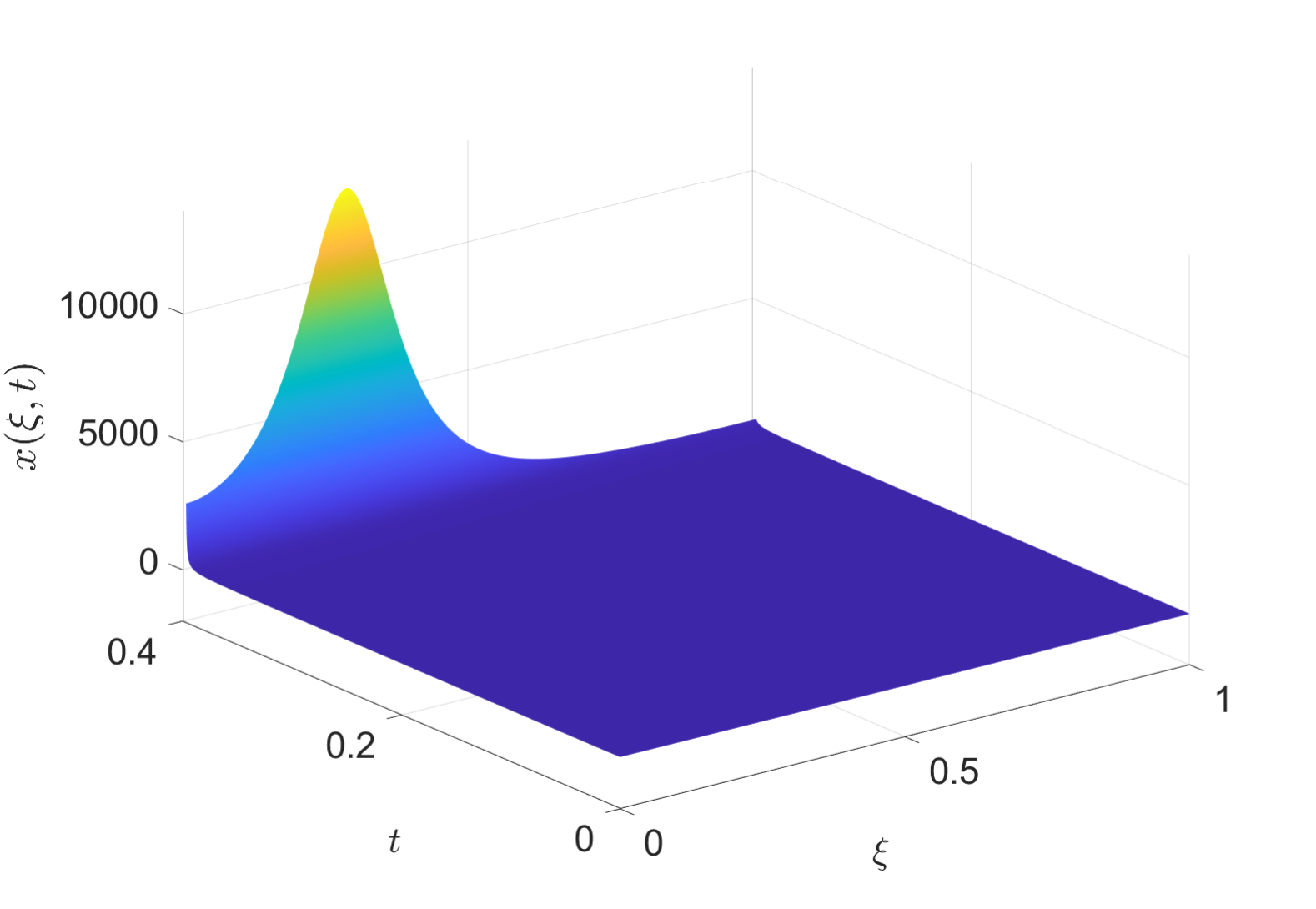}}\hspace{50pt}
		\subfigure[{Evolution of $\|x\|$.}]{ \includegraphics[width=0.35\textwidth,height=7pc]{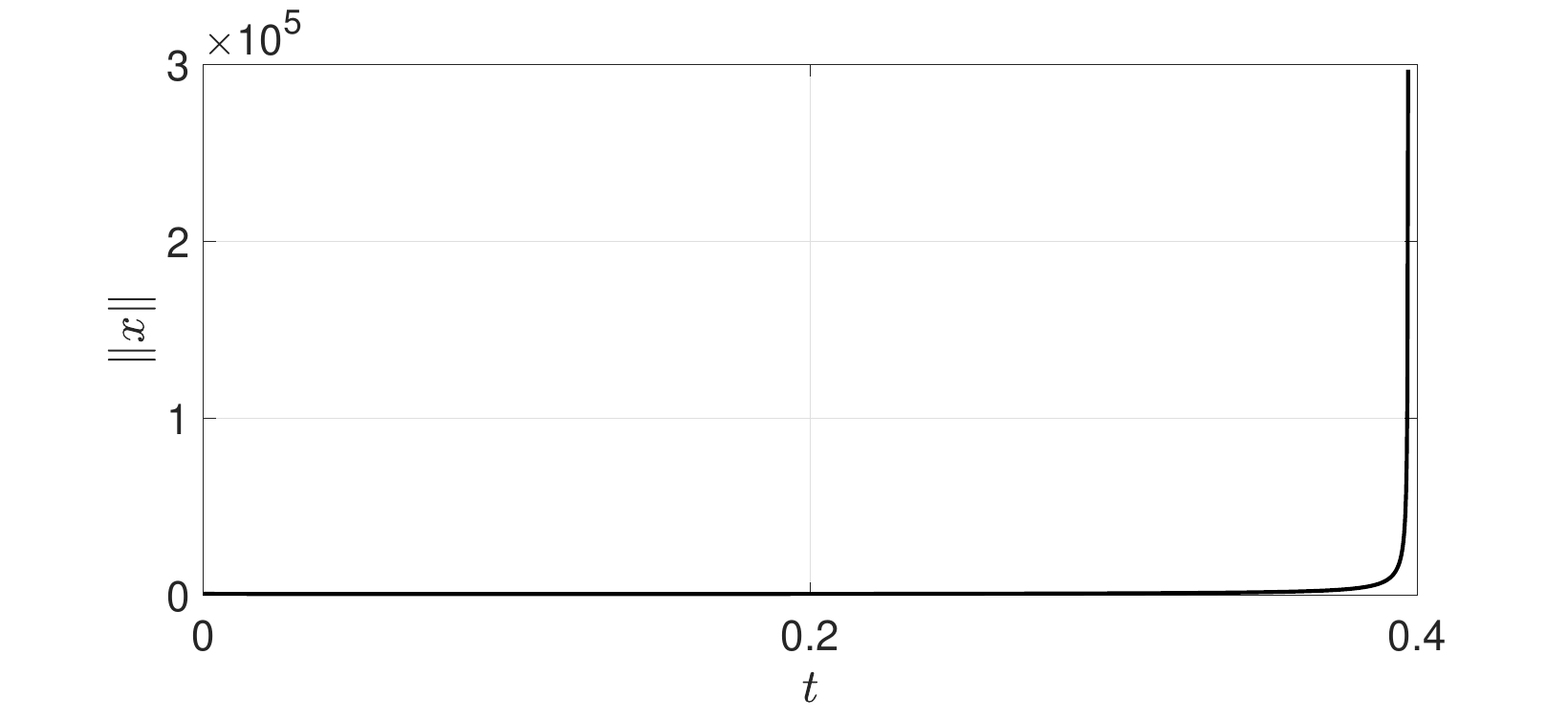}}
		\caption{Evolution of $x$   for system \eqref{4.4} {in the presence of disturbances with  large initial data $x_0=23(\xi-0.5)(\xi-2)$.}\label{fig6}}
	\end{center}
\end{figure*}

{{Figure \ref{fig6} illustrates that,} for relatively large initial data, the  {solution to system \eqref{4.4}} increase rapidly in  finite time, indicating  {system \eqref{4.4} may blow up}. Figure \ref{fig6`}  shows that, {for  relatively} small initial data, the {disturbance-free} system \eqref{4.4}   {is asymptotically stable at the origin}. Meanwhile, 
{in the presence of different disturbances, Fig.  \ref{fig7}, Fig. \ref{fig8}, and Fig. \ref{fig9} demonstrate that the solution {to} system \eqref{4.4} remains bounded for small initial data under different disturbances. Especially, Fig. \ref{fig7}(a) and (b)(or Fig.~\ref{fig8}(a) and (b)) show that, for the same small initial data, the  amplitude of the states of system \eqref{4.4} decreases as the amplitude of the disturbances diminishes. A comparison between Fig. \ref{fig7}(a) and Fig. \ref{fig8}(a) (or between Fig. \ref{fig7}(b) and Fig.\ref{fig8}(b)) shows that, for the same disturbance, the amplitude of the states of system \eqref{4.4} decreases as the initial datum becomes small. Figure \ref{fig9} shows the evolution of the state's norm of system \eqref{4.4} under different disturbances and initial data. In Fig. \ref{fig9}, all solid and dashed curves indicate that, for sufficiently large $t$, the ultimate bound of the state's norm is determined by the disturbances rather than the initial data. Thus, despite different initial data, the norms of the states admit    almost the same bound  when the time is sufficiently large, if the disturbances are identical. Overall, Fig. \ref{fig6}, Fig. \ref{fig6`}, Fig. \ref{fig7}, Fig\ref{fig8}, and Fig \ref{fig9} collectively well illustrate the LiISS property of system \eqref{4.4}.}

\begin{figure*}[htpb!]
	\begin{center}
		\subfigure[{Evolution of $x$ when $x_0=0.5(\xi-0.5)(\xi-2)$.}]{
			\includegraphics[width=0.35\textwidth]{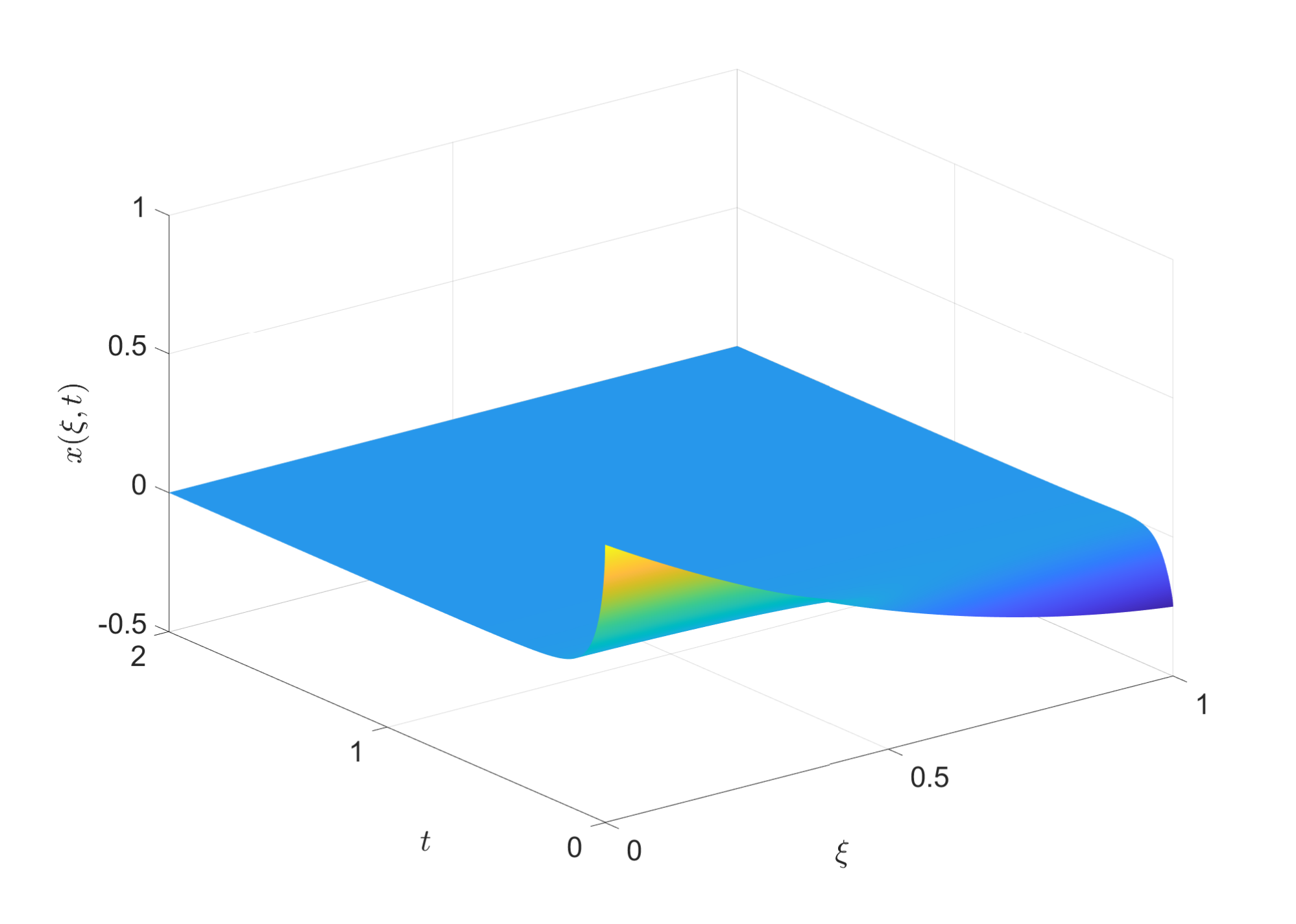}}\hspace{50pt}
		\subfigure[{Evolution of $x$ when   $x_0=0.8(\xi-0.5)(\xi-2)$.}]{
			\includegraphics[width=0.35\textwidth]{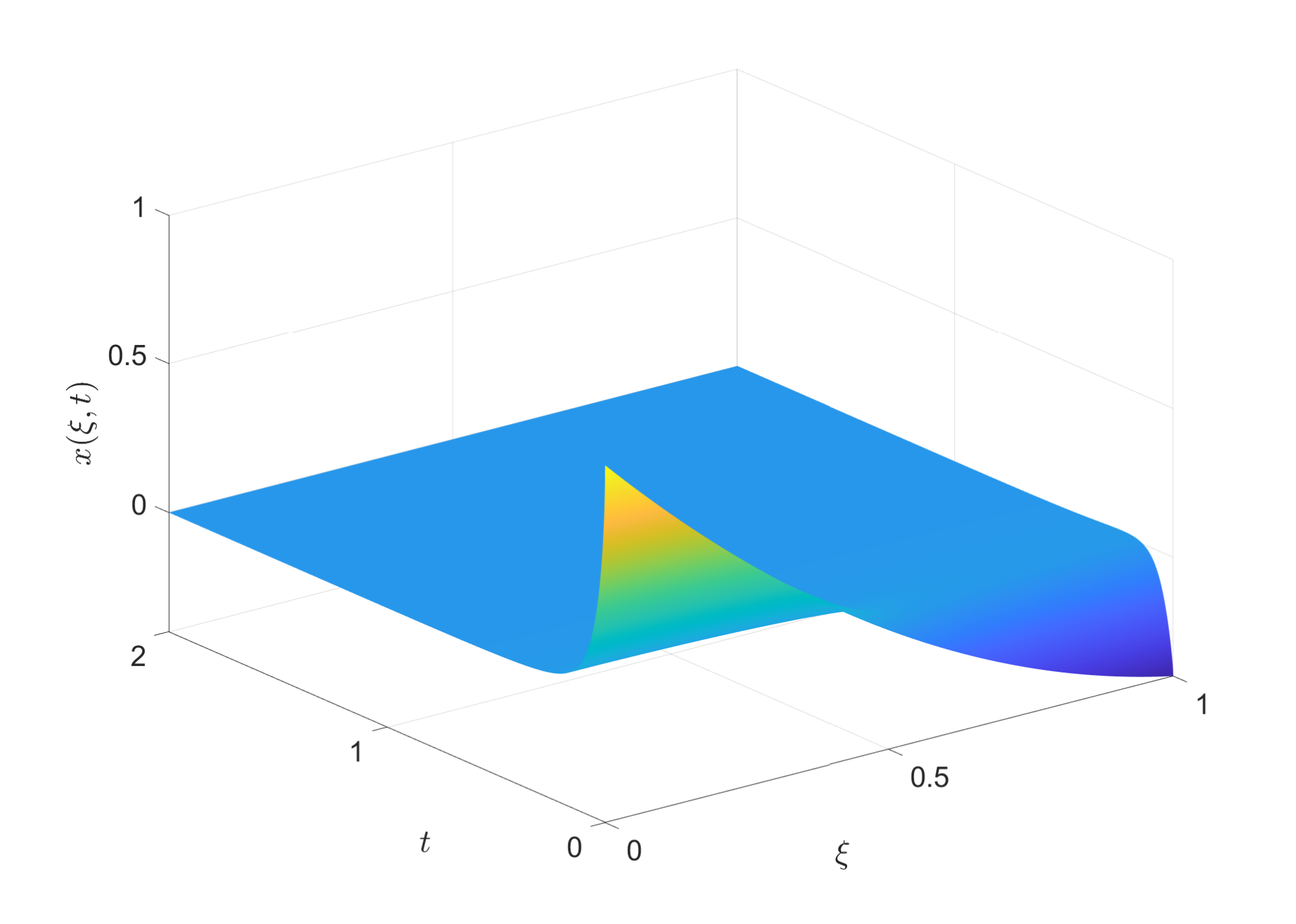}}
		\caption{Evolution of $x$   for {the disturbance-free} system \eqref{4.4}   {with different small initial data}.\label{fig6`}}
	\end{center}
\end{figure*}

\begin{figure*}[htpb!]
	\begin{center}
		\subfigure[{Evolution of $x$ {when     $A_1=1$}.}]{
			\includegraphics[width=0.35\textwidth]{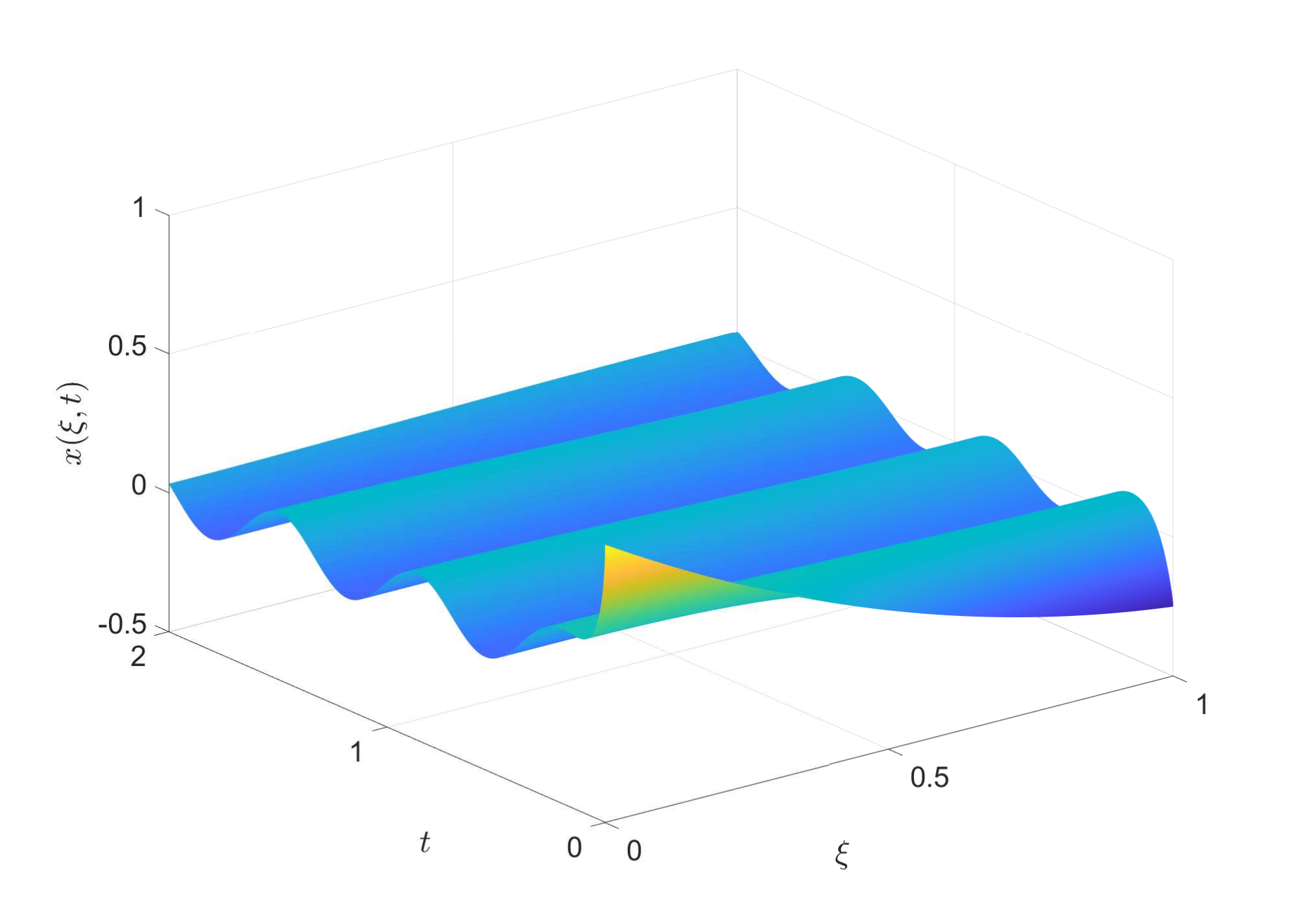}}\hspace{50pt}
		\subfigure[{Evolution of $ x $ when   { $A_1=3$}.}]{
			\includegraphics[width=0.35\textwidth]{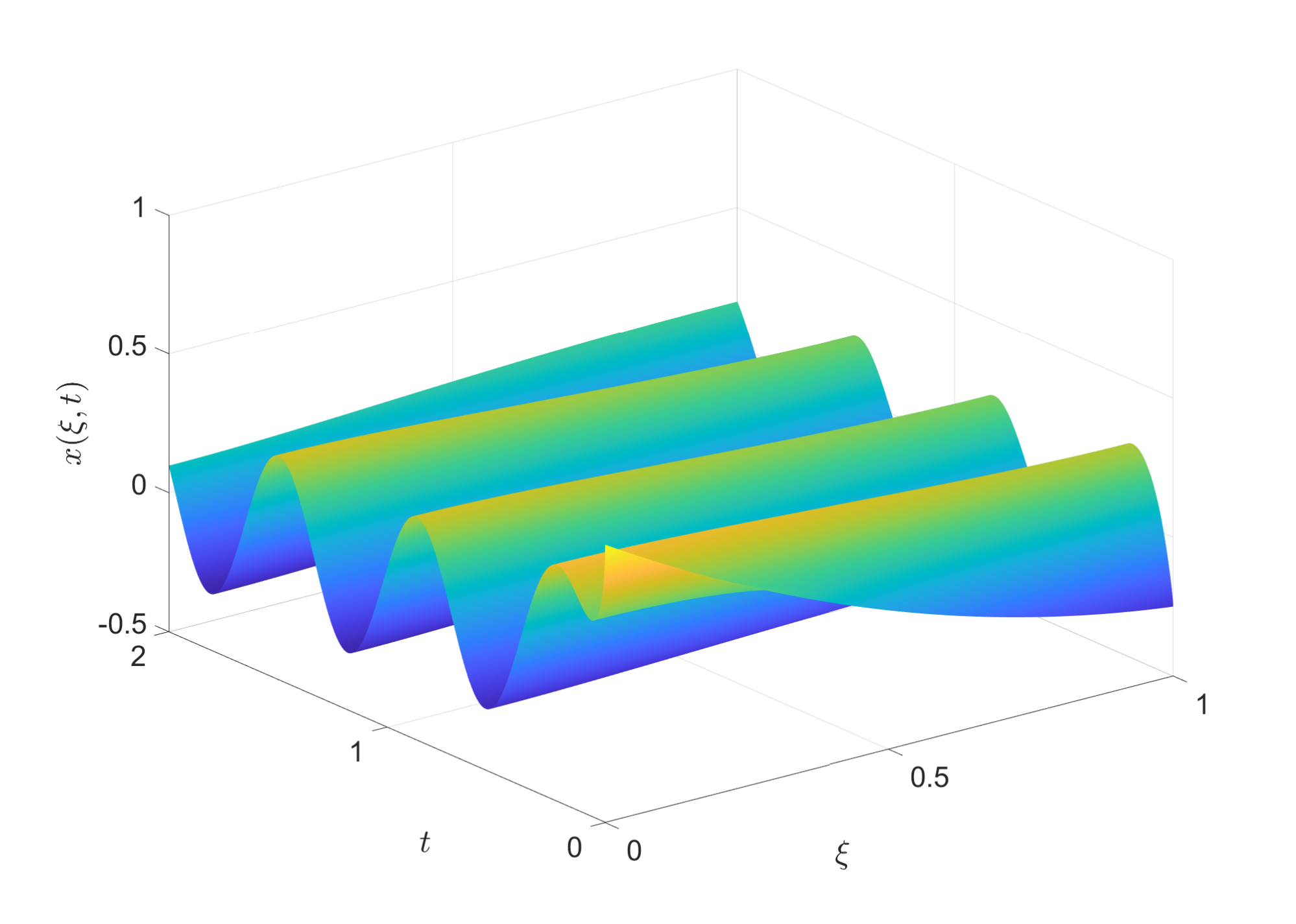}}
		\caption{Evolution of $x$   for system \eqref{4.4}  with different disturbances {when $x_0=0.5(\xi-0.5)(\xi-2)$}.\label{fig7}}
	\end{center}
\end{figure*}

\begin{figure*}[htpb!]
	\begin{center}
		\subfigure[{Evolution of $x$ when   $A_1=1$.}]{
			\includegraphics[width=0.35\textwidth]{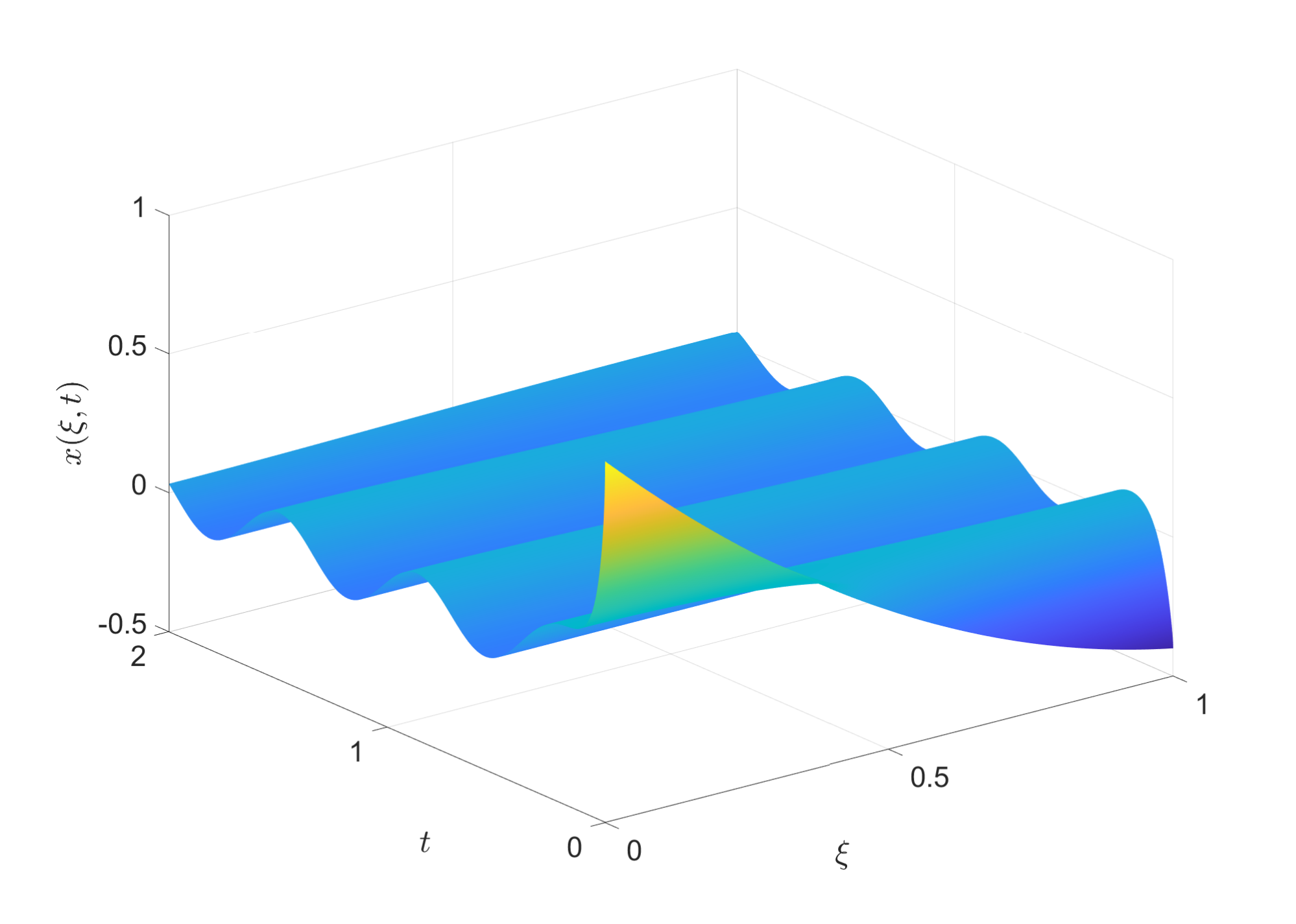}}\hspace{50pt}
		\subfigure[{Evolution of $ x $ when    $A_1=3$.}]{
			\includegraphics[width=0.35\textwidth]{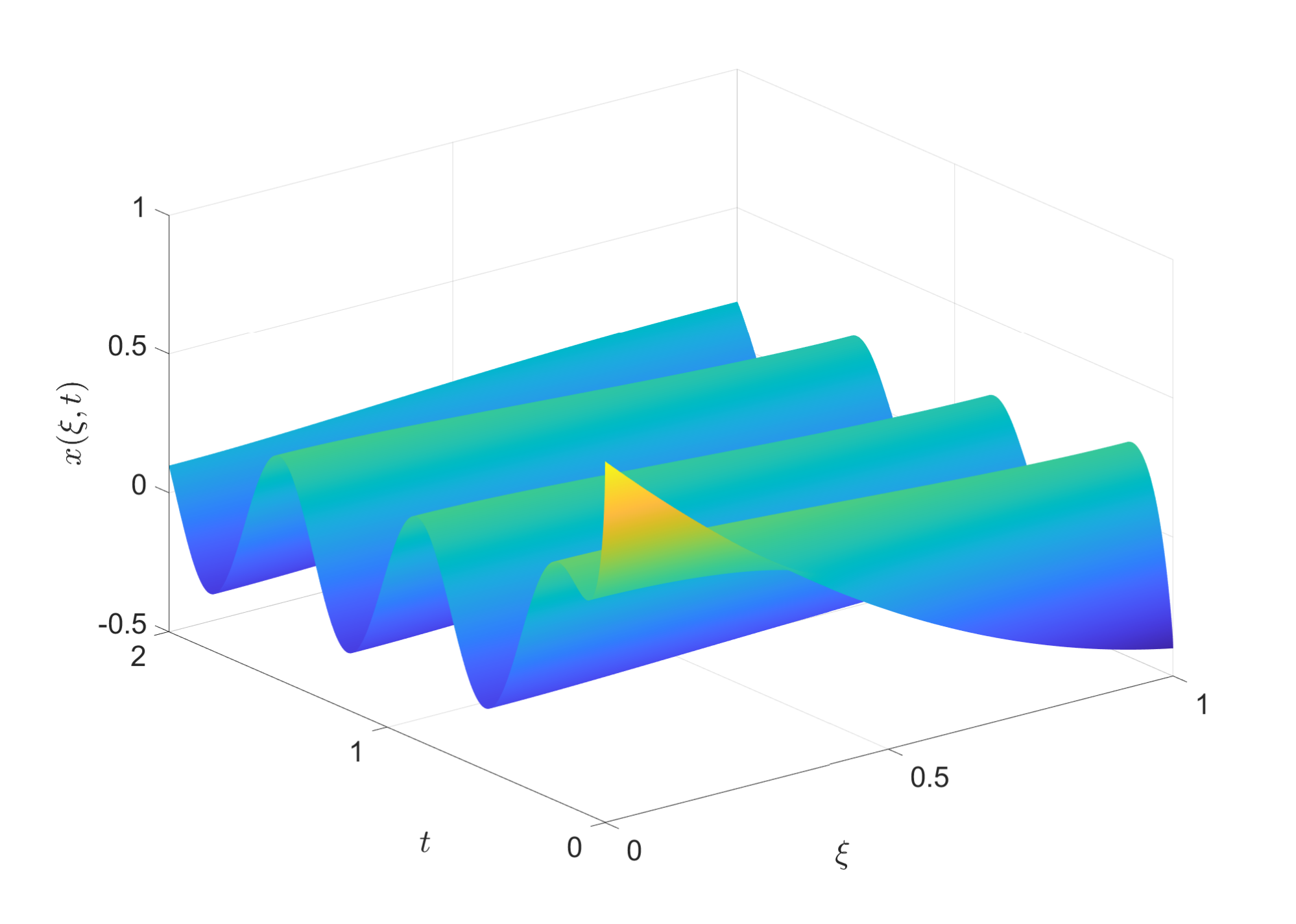}}
		\caption{Evolution of $x$   for system \eqref{4.4} with different disturbances when $x_0=0.8(\xi-0.5)(\xi-2)$.\label{fig8}}
	\end{center}
\end{figure*}
\begin{figure*}[htpb!]
	\begin{center}
		{ \includegraphics[width=0.35\textwidth,height=7pc]{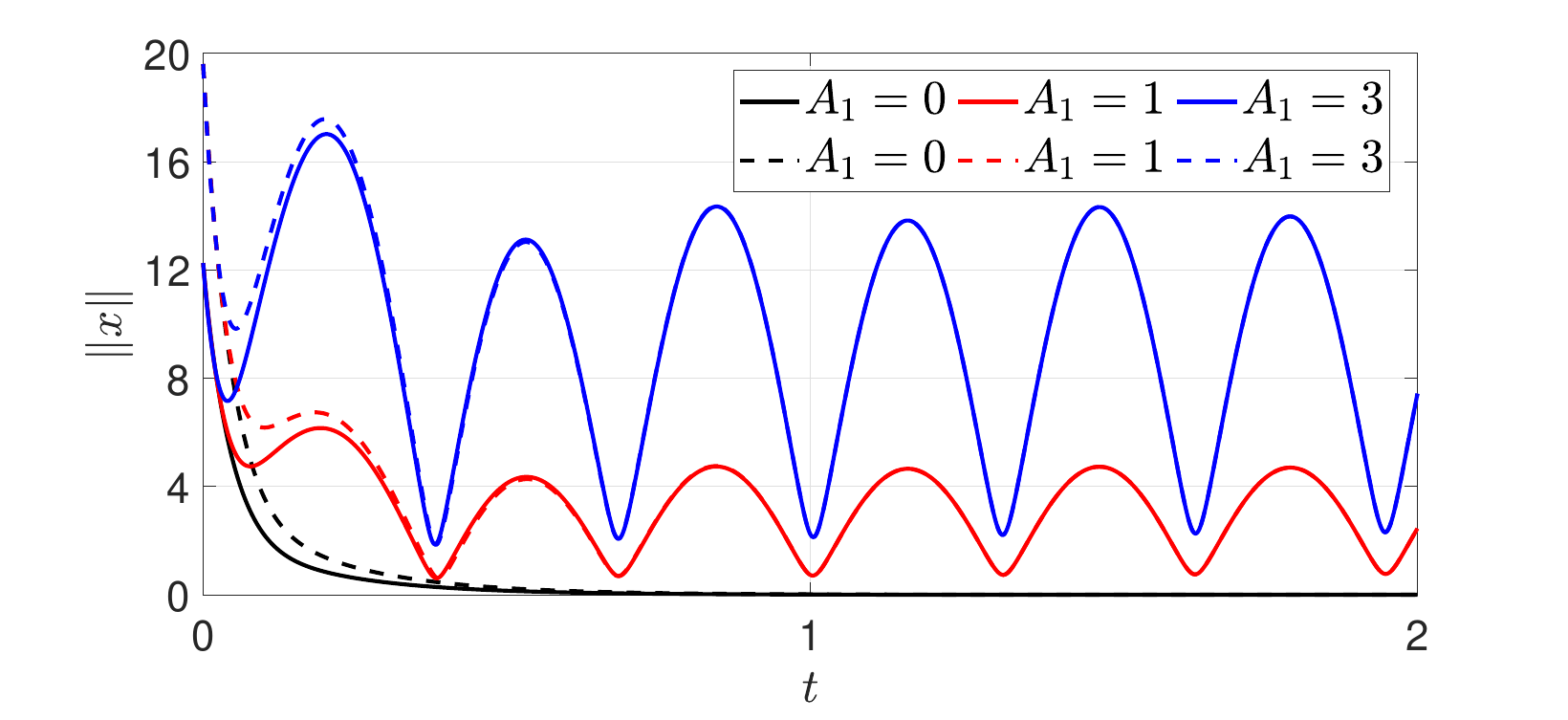}}
		\caption{{ Evolution of $\|x\|$ with different disturbances and initial data: solid {curve} for $x_0=0.5(\xi-0.5)(\xi-2)$ and  dashed {curve} for $x_0=0.8(\xi-0.5)(\xi-2)$.\label{fig9}}}
	\end{center}
\end{figure*}

\section{Conclusion}\label{conclusion}

In this work, we developed a set of  Lyapunov analytical tools for the  LiISS analysis of    non-autonomous infinite-dimensional systems  involving superlinear terms, corresponding to which   the nonlinear functionals  map the state from a small space to a large space. More precisely,  we  first provided several forms of generalized comparison principles for  nonlinear ODEs having time-varying coefficients, which enables us to conduct the LiISS analysis for non-autonomous infinite-dimensional systems  in the framework of Lyapunov stability theory, thereby avoiding the derivation of the {0-UASs (or 0-UGASs)}  via the semigroup theory.  Then, for non-autonomous infinite-dimensional systems,
we proved  an LiISS Lyapunov theorem   in the framework of Banach spaces and  provided sufficient conditions to ensure the existence of an LiISS-LF    in the framework of Hilbert spaces.
Through two examples, we illustrated how to use the developed result to verify the LiISS for a nonlinear finite-dimensional system   under linear state feedback control   and   a   nonlinear multi-dimensional parabolic PDE, respectively. In particular, we  demonstrated how to use the interpolation inequalities to handle superlinear terms in  the PDE and establish the LiISS by using the Lyapunov method.
It is noteworthy that for non-autonomous   boundary control  systems, constructing appropriate Lyapunov functionals  presents significant challenges.  Consequently, conducting the LiISS
analysis for such systems remains inherently difficult and will be investigated in future work.






\end{document}